\documentclass[a4paper,reqno]{amsart}
\usepackage{amscd, amssymb, pb-diagram, amsmath}
\usepackage{amsthm}
\usepackage{slashed,hyperref}

\usepackage{tikz}

\usepackage{tikz-cd}
\usepackage[margin=3.7cm]{geometry}
\usepackage{color}
\usepackage{tensor}
\usetikzlibrary{arrows,matrix,backgrounds,decorations.markings}
\numberwithin{equation}{section}

\usepackage[all]{xy}

\DeclareFontFamily{OT1}{pzc}{}
\DeclareFontShape{OT1}{pzc}{m}{it}{<-> s * [1.2] pzcmi7t}{}
\DeclareMathAlphabet{\mathpzc}{OT1}{pzc}{m}{it}

\def\Dom{\operatorname{Dom}}

\def\ker{\operatorname{Ker}}

\def\rd{\operatorname{d}\!}

\def\Tr{\operatorname{Tr}}

\def\sup{\operatorname{sup}}
\def\max{\operatorname{max}}
\def\min{\operatorname{min}}

\def\Im{\operatorname{Im}}

\def\Prob{\operatorname{Prob}}

\def\d{\operatorname{d}}
\def\ent{\operatorname{h}}
\def\t_c{\operatorname{t_{\text{c}}}}
\def\pre{\operatorname{P}}
\def\D_c{D_{\text{c}}}
\def\L_c{L_{\text{c}}}
\def\Res{\operatorname{Res}}

\def\df{\operatorname{d_f}}

\def\LC{\operatorname{C_{c}^\infty}(G_A)}

\def\C{\mathbb{C}}

\def\R{\mathbb{R}}
\def\N{\mathbb{N}}
\def\Z{\mathbb{Z}}
\def\T{\mathbb{T}}

\newtheorem{thm*}{Theorem}
\newtheorem{thm}{Theorem}[section]

\newtheorem{cor}[thm]{Corollary}

\newtheorem{lemma}[thm]{Lemma}
\newtheorem{prop}[thm]{Proposition}

\newtheorem{thma}{Theorem}

\newtheorem{problema}{Open problem}

\theoremstyle{definition}
\newtheorem{definition}[thm]{Definition}

\theoremstyle{remark}
\newtheorem{remark}[thm]{Remark}

\newtheorem{example}[thm]{Example}

\makeatletter
\@namedef{subjclassname@2020}{%
  \textup{2020} Mathematics Subject Classification}
\makeatother

\begin{document}


\title[Five shades of KMS]{Five shades of KMS: Statistical properties in the spectral geometry of Cuntz--Krieger algebras}

\author{Dimitris M. Gerontogiannis}
\address{Institute of Mathematics of the Polish Academy of Sciences, ul. {\'S}niadeckich 8, 00–656, Warszawa, Poland}
\email{dgerontogiannis@impan.pl}

\author{Magnus Goffeng}
\address{Centre for Mathematical Sciences, Lund University, Box 118, SE-221 00, Lund, Sweden}
\email{magnus.goffeng@math.lth.se}

\begin{abstract}
We study spectral invariants arising in the noncommutative geometry of topological Markov chains and Cuntz--Krieger algebras. Their noncommutative geometry is described by spectral triples built from log-Laplacians, which are known to have non-trivial index theory and exotic quantum symmetries. We prove statistical eigenvalue asymptotics, local heat trace asymptotics,
local Weyl laws, and an analogue of Connes' trace theorem. In all cases the local asymptotics are governed by the Kubo--Martin--Schwinger state of the gauge action on the Cuntz--Krieger algebra. 
\end{abstract}



\maketitle
\vspace{-0.9cm}

\tableofcontents
\vspace{-0.9cm}

\section{Introduction}

An idea going back to Weyl's law from 1911 is that the spectral properties of geometric operators encode the underlying geometry. Kac's seminal question ``can one hear the shape of a drum?'' from 1966 neatly summarises the mathematical study of spectral geometry \cite{topics}. This field has seen numerous highlights, such as the sharp Weyl law \cite{hormander68}, Cheeger's inequality and its importance to expander graphs \cite{lubotzky}, Selberg's trace formula \cite{hejhal}, as well as quantum ergodicity and quantum chaos \cite{zelditchqe}.

In parallel to spectral geometry, the development of Connes' program \cite{connesredbook} for spectral noncommutative geometry took place. This program aims to extend classical notions such as differential geometry, metric geometry, measure theory and dynamics to the broader setting of operator algebras through the lens of spectral geometry. The intersection of spectral noncommutative geometry and classical spectral geometry appears in the study of manifolds. In the full generality of the former, it seems futile to hope for as refined results as in the classical setting of manifolds. Nevertheless, in examples, one can hope to retrieve new information about noncommutative spaces by paraphrasing classical results. 

The purpose of this paper is to do this for the spectral noncommutative geometry of Cuntz--Krieger algebras as set up in \cite{GGM25}. Namely, we prove eigenvalue asymptotics (in a statistical sense) and compute local heat trace asymptotics, as well as a local Weyl law. Interestingly, this turns out to be far more than a mere exploration, as these results culminate in intrinsic descriptions of the KMS-state of the gauge action. Moreover, our spectral invariants recover dynamical properties of the associated subshift, for instance entropy, counting functions for triangles in the hyperbolic tree of finite paths, and can be used to prove an averaged form of quantum ergodicity.

Cuntz--Krieger algebras are operator algebras encoding topological Markov chains, and more generally, subshifts of finite type. They were introduced by Cuntz--Krieger in \cite{CK}, and serve as prototypical examples of highly noncommutative $C^*$-algebras. More precisely, such a $C^*$-algebra is generated by the continuous functions on the infinite path space of the subshift, the transfer operator together with the shift. A geometrically more tractable description is given by the groupoid $C^*$-algebra of the Deaconu--Renault groupoid of the subshift. Its $K$-homology recovers dynamical invariants such as the Bowen--Franks invariant \cite{CK2}. Moreover, Cuntz--Krieger algebras satisfy $KK$-duality \cite{KPDual} ensuring that, topologically, their noncommutative geometry behaves like a finite CW-complex. Nevertheless, Cuntz--Krieger algebras are purely infinite and therefore their noncommutative geometry must be infinite-dimensional \cite{connes89,GMCK,GRU}. As such, the spectral noncommutative geometry of Cuntz--Krieger algebras is highly non-classical compared to \cite{pongetian}.

The construction of spectral noncommutative geometry on Cuntz--Krieger algebras has been pursued via algebraically flavoured methods in \cite{GMCK,GMR}, and more analytically in \cite{GGM25}. Particular attention has been devoted to dimensional features (finite summability), topological features (index theory), and geometric ones (isometry groups). In all those cases, the noncommutative geometries exhaust the range of the Bowen--Franks invariant (i.e. odd $K$-homology) via abstract elliptic operators. Partial results have also been obtained for understanding the connection to equilibrium and KMS-states \cite{GRU}. We also mention related work in the commutative case of the continuous functions on the infinite path space \cite{farsietal,antoineputnam, Sharp}. Further, the noncommutative geometry approach \cite{GGM25} to topological Markov chains found new impetus in the recent work \cite{FGM}, where the quantum isometry group of Cuntz--Krieger algebras was computed. This led to the first ergodic action of a compact matrix quantum group on the Cantor space.

Here, we take our starting point in the analytic approach \cite{GGM25} based on the log-Laplacian \cite{GerMes}. Compared to the earlier constructions \cite{GMCK}, this approach appears particularly significant because the setup in \cite{GGM25} allows for computations using heat kernel methods, spectral asymptotics, and geometric invariants analogous to those appearing in classical spectral geometry.

\subsection{Setup and background}
We consider a topological Markov chain defined from a primitive matrix $A\in M_N(\{0,1\})$. Before going into the details of our results, let us describe the setup and provide further notation. For a more detailed presentation we refer to \cite{GGM25} and, when relevant, to later sections of the current paper. 

We write $\Omega_{A}$ for the associated one-sided shift space
\[\Omega_A=\{x_1x_2x_3\cdots\in \{1,\dots,N\}^{\mathbb N} : A_{x_nx_{n+1}}=1,\, n\in \mathbb N  \},\]
with shift map $\sigma(x_1x_2x_3\cdots)=x_2x_3\cdots$. Also, $V_{A}$ is the space of finite, admissible words $V_A:=\bigcup_k\{x_1\cdots x_k\in \{1,\dots,N\}^{k} \mid A_{x_nx_{n+1}}=1,\, 1\leq n\leq k-1 \}$ and the length of $w\in V_A$ is denoted by $|w|$. The empty word $\o$ has length $|\o|=0$. For more details, see Subsection \ref{topmark}. 

The Cuntz--Krieger algebra $O_A$ is the universal $C^*$-algebra generated by
elements $S_1,\dots,S_N$ satisfying
\[\sum_{i=1}^N S_iS_i^*=1,\quad\mbox{and}\quad S_i^*S_k=\delta_{i,k} \sum_{j=1}^N A_{ij}S_jS_j^*. \]
These relations encode the admissible transitions of the subshift and this is more visible in the groupoid model $C^*_r(G_A)$ of $O_A$, see Subsection \ref{subsec:drgroup}. Further, finite words $\alpha=\alpha_1\cdots \alpha_n\in V_A$ define elements $S_\alpha=S_{\alpha_1}\cdots S_{\alpha_n}$. The projections $S_\alpha S_\alpha^*$ correspond to the characteristic functions of the clopen cylinder sets $C(\alpha):=\{\alpha x_{n+1}x_{n+2}\cdots \in \Omega_A\}$ that form a basis for the topology of $\Omega_A$. The commutative subalgebra generated by these projections is canonically isomorphic to the $C^*$-algebra of continuous functions $C(\Omega_A)$, which forms a maximal abelian subalgebra of $O_A$. For more details see Subsection \ref{subsec:dkal}.

From the perspective of dynamical systems, it is interesting to study the equilibrium states \cite{ruelle}. In the operator algebra literature \cite{lacanesh}, equilibrium states often correspond to KMS-states at the operator algebra level. In our setting, there is a natural action $\alpha$ of $\T$ on the algebra $O_A$, called the gauge action, given on generators by
\[\alpha_z(S_i)=zS_i, \qquad z\in \T. \]
A state $\varphi$ on $O_A$ satisfies the KMS condition at inverse temperature
$\beta>0$ if
\[
\varphi(ab)=\varphi\!\left(b\,\alpha_{\mathrm{e}^{-\beta}}(a)\right)
\]
for elements $a,b$ analytic for the action $\alpha$. In general, for irreducible shifts, KMS-states are in one-to-one correspondence with equilibrium measures in thermodynamic formalism. In our primitive setting, the Perron--Frobenius--Ruelle transfer operator
\[
(\mathcal L f)(x)
=
\sum_{\sigma(y)=x} f(y)
\]
determines a distinguished equilibrium (conformal) measure, and the associated KMS-state $\varphi_A$ on $O_A$ is obtained from integration against that measure. More precisely, there is a unique KMS-state for the gauge action and it has inverse temperature $\beta=\log\lambda_A>0$, where $\lambda_A$ is the Perron--Frobenius eigenvalue of $A$.

Passing to noncommutative geometry \cite{connesredbook}, spectral triples are expected to provide a geometric description of Cuntz--Krieger algebras and their dynamics. Namely, a spectral triple $(\mathcal A,\mathcal H,D)$ on $O_A$ consists of a dense $*$-subalgebra $\mathcal A\subseteq  O_A$, a Hilbert space $\mathcal H$ on which $O_A$ is represented, and an unbounded self-adjoint operator $D$ with compact resolvent such that, for all $a\in \mathcal A$, we have that $a\Dom(D)\subseteq \Dom D$ and $[D,a]$ extends to a bounded operator. In the commutative setting of a compact manifold, one can take $D$ as a self-adjoint, elliptic, first order differential operator (e.g. a Dirac type operator), encoding both metric, topological and measure-theoretic information via spectral geometry \cite{connesredbook}. The spectral triples on $O_A$ of relevance in this paper are those of \cite{GGM25}, which use the conformal structure on the subshift $\Omega_A$ and the associated Deaconu--Renault groupoid $G_A$, described in Subsection \ref{subsec:drgroup}. 

\enlargethispage{\baselineskip}
\enlargethispage{\baselineskip}

As the discussion above indicates, Cuntz--Krieger algebras form a bridge between symbolic dynamics, operator algebras, statistical mechanics, and noncommutative geometry. Topological Markov chains supply the combinatorial data, KMS-states encode equilibrium phenomena, and spectral triples furnish a geometric interpretation of the resulting noncommutative spaces. The goal of this paper is to fully clarify these connections by showing that the spectral noncommutative geometry of \cite{GGM25} reproduces dynamical features governed by thermodynamic formalism.

\subsection{Main results}

We study the spectral triple $(\LC,L^2(G_A),D)$ of $O_A$ from \cite{GGM25}. Here $G_A$ denotes the Deaconu--Renault groupoid of the shift map $\sigma:\Omega_A\to \Omega_A$. As discussed in more detail in Subsection \ref{subsec:ham}, we decompose $G_A=\bigsqcup_{\gamma\in I_A} G_\gamma$ into particular bisections indexed by $I_A\subseteq V_A\times V_A$, through which we lift the conformal structure from $\Omega_A$ to $G_A$. Then, we write 
$$L^2(G_A)=\bigoplus_{\gamma\in I_A} L^2(G_\gamma).$$
On each $L^2(G_\gamma)$ we have the log-Laplacian $\Delta_\gamma$ and set $\Delta:=\bigoplus_\gamma \Delta_\gamma$. The Hamiltonian type operator $D$ is given by 
$$D:=-\Delta+(2P-1)M_L,$$
where $P$ is a projection with $\Im P\subset \ker \Delta$ and $M_L$ is the multiplication operator by a locally constant function $L$ on $G_A$. Also, $\LC$ denotes the convolution $*$-subalgebra of $C_r^*(G_A)$ of locally constant compactly supported continuous functions on $G_A$. It holds that the commutator $[D,f]$ extends to a bounded operator on $L^2(G_A)$, for all $f\in \LC$.

Our main results concern the eigenvalue asymptotics of $D$, the asymptotic behaviour of local heat traces $t\mapsto \Tr(a\mathrm{e}^{-t|D|})$ and local counting functions $\theta\mapsto \Tr(a\chi_{[0,\theta]}(|D|))$ for $a\in \LC$, where $\chi_{[0,\theta]}(|D|)$ are the spectral projections of $|D|$. \\

We begin by discussing our results on eigenvalue asymptotics. The reader can find more details in Section \ref{secstatarad}. It was shown in \cite{GGM25} that $L^2(G_A)$ admits a particular ON-basis of eigenfunctions for $\Delta$ taking the form $(\mathrm{e}_{\gamma},\mathrm{e}_{(\gamma,\nu,j)})_{(\gamma,\nu,j)\in \mathfrak{I}_A}\subset \LC$, for an appropriate index set $\mathfrak{I}_A\subseteq I_A\times V_A\times \{1,\ldots,N\}$. In $\mathfrak{I}_A$, the parameter $\gamma$ ranges over $I_A$. Also, at this point it is important to note that $I_A$ admits range and source maps, namely $r:I_A\to V_A$ and $s:I_A\to V_A\setminus \{\o\}$. Now, for each $\gamma$, the collection $(\mathrm{e}_{\gamma},\mathrm{e}_{(\gamma,\nu,j)})_{(\nu,j)}$ forms an ON-basis for $L^2(G_\gamma)$ out of eigenfunctions of $\Delta_{\gamma}$ and $\nu$ ranges over $s(\gamma)V_A$. Moreover, the kernel of $\Delta$ is spanned by $(\mathrm{e}_{\gamma})_{\gamma\in I_A}$. Also, the non-zero eigenvalues form $(\lambda^A_{s(\gamma)}(\nu))_{(\gamma,\nu,j)\in \mathfrak{I}_A}$ with $\Delta\mathrm{e}_{(\gamma,\nu,j)}=\lambda^A_{s(\gamma)}(\nu)\mathrm{e}_{(\gamma,\nu,j)}$. 

Therefore, in order to understand the eigenvalue asymptotics of $D$ we first describe those of each $\Delta_\gamma$. Specifically, for fixed $\gamma$, the non-zero eigenvalues are $(\lambda^A_{s(\gamma)}(\nu))_{\nu\in s(\gamma)V_A}$ with $\lambda^A_{s(\gamma)}(s(\gamma))=\lambda_{A}u_{s(\gamma)_{|s(\gamma)|}}$, and if $\nu\in s(\gamma)V_A \setminus \{s(\gamma)\}$ then
$$\lambda_{s(\gamma)}^A(\nu)=\lambda_{A}u_{\nu_{|\nu|}}+\lambda_{A}\sum_{k=0}^{|\nu|-|s(\gamma)|-1} u_{\nu_{|s(\gamma)|+k}}(1-P_{\nu_{|s(\gamma)|+k},\nu_{|s(\gamma)|+k+1}}).$$
Here $u=(u_i)_{i=1}^N$ is the $\ell^1$-normalized Perron--Frobenius eigenvector and $P_{ij}$ are the transition probabilities in the graph defined from $A$, see Equation \eqref{transprob}. Except in certain regular cases, the eigenvalues $\lambda^A_{s(\gamma)}(\nu)$ do not appear to admit a simpler structural description beyond upper and lower bounds by multiples of $|\nu|$. However, we can describe their statistics as follows. For $n\in \mathbb N$ consider the projection onto the first $n$ terms $\pi_n:\Omega_{A}\to V_{A}^n$. Consider the potential function $F_A:\Omega_{A}\to [0,1)$ given by 
$$F_A(x):=u_{x_1}(1-P_{x_{1}, x_{2}}).$$

\begin{thma}
\label{thmA}
Assume that $\tau$ is a $\sigma$-invariant ergodic measure on $\Omega_A$. Then, for every $\gamma\in I_A$ and $\tau$-almost every $x\in C(s(\gamma))$ we have 
$$\lim_{n\to \infty} \frac{\lambda_{s(\gamma)}^{A}(\pi_n(x))}{n}=\lambda_{A} \tau(F_A).$$ 
Moreover, there are $0<c_0\leq c_1<1$ independent of $\tau$ such that $c_0\leq \tau(F_A)\leq c_1 \lambda_A^{-1}$. If $\tau$ is the Parry measure, then $\tau(F_A)$ is the total Gini impurity\footnote{The same statistical quantity used to measure species diversity in ecology.} of the graph given by $A$. Also, the graph is out-regular\footnote{I.e. the row-sums of $A$ are constant} if and only if $F_A$ is the constant function $N^{-1}(1-\lambda_A^{-1}).$
\end{thma}

The first statement of Theorem \ref{thmA} appears as Theorem \ref{lem:eigen_asym} below, while the remaining, more general statements are proved in Section \ref{secstatarad}.\\

We are also interested in studying the heat operator $\mathrm{e}^{-t|D|}$. Since we have built $\Delta$ as a ``log-Laplacian'', it is natural to ask if there is in fact anything logarithmic about it. In other words, is the exponential $\mathrm{e}^{-t|D|}$ given in terms of Riesz potential operators, which in the Euclidean setting are inverse powers of fractional Laplacians? We prove this in the special case that the graph of $A$ is out-regular.

\begin{thma}
\label{thmB}
Assume that all row-sums of $A$ are equal to $d\geq 2$. We set $c=(d-1)N^{-1}$, $\delta=c\log_{\lambda}(\mathrm{e})$, and $\df=\log_{\lambda} d$. For $t>\df \delta^{-1}$, the heat operator $\mathrm{e}^{-t|D|}$ admits a kernel 
$$K_t\in L^1(G_A\times G_A,\mu_{G_A}\times \mu_{G_A})\cap C(G_A\times G_A).$$ 
Specifically, for $g_1\neq g_2\in G_A$, it is given by 
$$K_t(g_1,g_2)=\begin{cases}\mathrm{e}^{-tc|\gamma|}k_{\gamma,t}(g_1,g_2), & \text{if  }\; g_1,g_2\in G_{\gamma}\\
0, & \text{otherwise}\end{cases}.$$ 
Here $k_{\gamma,t}\in L^1(G_{\gamma}\times G_{\gamma},\mu_{\gamma}\times \mu_{\gamma})\cap C(G_{\gamma}\times G_{\gamma})$ for $t>\df\delta^{-1}$, is the kernel of the heat operator $\mathrm{e}^{-t\Delta_{\gamma}}$ on each bisection $G_{\gamma}$ and is given by 
$$k_{\gamma,t}(g_1,g_2)=h_\gamma(t)+H_\gamma(t)\mathrm{d}_\gamma(g_1,g_2)^{-\df+\delta t},$$
where
\begin{align*}
h_\gamma(t)&=d^{|s(\gamma)|-1}N\left(1-\frac{(d-1)\mathrm{e}^{-dN^{-1}t}}{d\mathrm{e}^{-ct}-1}\right)\\
H_\gamma(t)&=\frac{N-N\mathrm{e}^{-ct}}{d\mathrm{e}^{-ct}-1}\mathrm{e}^{-t(dN^{-1}-c|s(\gamma)|)}.
\end{align*}

\end{thma}

The reader can find the heat kernel computation for the bisections in Theorem \ref{thm:Heat_kernel_formula}, which uses the closed formula for the eigenfunctions from Proposition \ref{prop:eigenfunctions}. The conclusion for $\mathrm{e}^{-t|D|}$ is in Proposition \ref{heat_kernel_formula_groupoid}.

\begin{problema}
It poses an interesting open problem to compute the heat kernel of $\mathrm{e}^{-t|D|}$ for general $A$.
\end{problema}

\begin{remark}
We note that 
\begin{equation}
\label{etdeltaon}
\mathrm{e}^{-t\Delta_\gamma}=d^{-|s(\gamma)|+1}N^{-1}h_\gamma(t)P_{\Delta_\gamma}+H_\gamma(t)R_{\gamma,\delta t},
\end{equation}
where $P_{\Delta_\gamma}$ denotes the projection onto the kernel of $\Delta_\gamma$ and $R_{\gamma,\delta t}$ is the Riesz potential acting on $L^2(G_{\gamma},\mu_{\gamma})$ by $$R_{\gamma,\delta t}f(g_1)=\int_{G_{\gamma}}\frac{f(g_2)}{d_{\gamma}(g_1,g_2)^{\df-\delta t}} \d \mu_{\gamma}(g_2).$$
In particular, the heat operator of the log-Laplacian is up to a finite rank operator a Riesz potential operator.\\
\end{remark}

A large portion of the paper is spent on the asymptotic behaviour of heat traces and counting functions, results pertaining to the spirit of spectral geometry in spectral noncommutative geometry. The main goal of the next theorem is to connect these classical notions with thermodynamic formalism through the KMS-state. A crucial feature is the following \textit{critical time} $\t_c\in (0,\infty)$ associated to $(\LC,L^2(G_A),D)$, given by 
$$\t_c:=\lambda_{A}^{-1}\sup\limits_{\tau}\frac{\ent_{\tau}}{\tau(F_A)},$$
where the supremum is taken over all $\sigma$-invariant ergodic probability measures $\tau$ on $\Omega_A$, and $\ent_{\tau}$ is the measure entropy of $\tau$. This critical time is first observed in Proposition \ref{prop:Laplacian_est}. For out-regular graphs, we have that $\t_c=\df\delta^{-1}$, the threshold appearing in Theorem \ref{thmB}.

\begin{thma}
\label{thmC}
The KMS-state $\varphi_A$ on $O_A$ can be reconstructed from the spectral triple $(\LC,L^2(G_A),D)$ in any of the following ways:
\begin{enumerate}
\item Let $P_D$ denote the non-negative spectral projection of $D$. For any $a\in \LC$, the local positive part heat traces $t\mapsto \Tr(P_Da\mathrm{e}^{-tD})$ are finite for $t>\t_c$ and extend to meromorphic functions of $t\in \C$ with at most order $N$ poles, situated in a specific discrete set independent of $a$. Moreover, the family of states $$\psi_t(a):=\frac{\Tr(P_D a\mathrm{e}^{-tD})}{\Tr(P_D\mathrm{e}^{-tD})},$$ weak$^*$ converges to $\varphi_A$, as $t\to \t_c^+$.
\item For any $a\in \LC$, the local heat traces $t\mapsto \Tr(a\mathrm{e}^{-t|D|})$ are finite for $t>\t_c$ and extend to meromorphic functions of $t\in \C$ with at most order $3N$ poles, situated in a specific discrete set independent of $a$. Moreover, the family of states 
$$\varphi_t(a):=\frac{\Tr(a\mathrm{e}^{-t|D|})}{\Tr(\mathrm{e}^{-t|D|})},$$
weak$^*$ converges to $\varphi_A$, as $t\to \t_c^+$.
\item For any dilation invariant extended limit $\omega$ on $\ell^{\infty}(\N)$ with associated Dixmier trace $\Tr_{\omega}$, and for every $a\in \LC$ it holds that 
$$\varphi_A(a)=\frac{2\t_c}{C}\Tr_{\omega}(a |D|^{-2}\mathrm{e}^{-\t_c |D|}).$$ 
Here $C:=\lim_{t\to \t_c^+} (t-\t_c)^3\Tr(\mathrm{e}^{-t|D|})>0$ is explicitly described in Theorem \ref{thm:HeatTr}. 
\item For $a\in \LC$, we can define the twisted $\zeta$-function
$$\zeta_D(s;a):=\mathrm{Tr}(a|D|^{-s}\mathrm{e}^{-\t_c |D|}), \quad \mbox{holomorphically for}\quad \mathrm{Re}(s)>3.$$ 
Then $\zeta_D(\cdot\,;a)$ extends to a meromorphic function in $\C$ whose poles lie at the points $s\in \{1,2,3\}$ and are of order most $1$. Moreover,
$$\varphi_A(a)=\frac{2}{C}\mathrm{Res}_{s=3}\zeta_D(s;a).$$ 
\end{enumerate}
If the row-sums and column-sums of $A$ are both uniformly constant (i.e. the graph is out-regular and in-regular), we also have that:
\begin{enumerate}
\item[(5)] for any $a\in \LC$, it holds 
$$\varphi_A(a)=\lim_{\theta\to \infty} \frac{\Tr(\chi_{[0,\theta]}(|D|) a)}{\Tr(\chi_{[0,\theta]}(|D|))}.$$
\end{enumerate}
\end{thma}

For all the statements of Theorem \ref{thmC}, the reader can find more refined statements concerning asymptotic behaviour in the body of the text. Item (1) of Theorem \ref{thmC} is based on the computation of the KMS-state of the Heisenberg flow associated with $D$, see Proposition \ref{prop:KMS}. Item (2) is based on the heat trace asymptotics of $\Tr(\mathrm{e}^{-t|D|})$ and its meromorphic extension studied in Subsection \ref{heattracegroupdopoad}. In particular, by the local heat trace asymptotics of $\Tr(a\mathrm{e}^{-t|D|})$ in Theorem \ref{thm:HeatTrwitha} we see that $\Tr(a\mathrm{e}^{-t|D|})$ and $\varphi_A(a)\Tr(\mathrm{e}^{-t|D|})$ coincide at leading order. Item (3) can be found as Theorem \ref{thm:Dixmier} in the body of the text. Item (4) follows from Item (2) and (3) as follows. Using that the order of the pole of $\Tr(\mathrm{e}^{-t|D|})$ at $t=\t_c$ is  $3$ (see item (3)) we have from item (2) that $\Tr(a\mathrm{e}^{-(t+t_c)|D|})=C\varphi_A(a)t^{-3}+O(t^{-2})$, and item (4) follows from standard techniques for Mellin transforms \cite[Proposition 5.1]{gs}.

Item (5) of Theorem \ref{thmC} is based on a sharp Weyl law for $|D|$, which we prove in Theorem \ref{thm:Weyl_law}. If all the row-sums and column-sums of $A$ are equal to $d\geq 2$, we show that the counting function 
$$\mathcal{N}(\theta):=\Tr(\chi_{[0,\theta]}(|D|))=\#\{\theta_j\leq \theta:j\geq 0\},$$ 
for the eigenvalues $\theta_j$ of $|D|$ (non-decreasing order, counting multiplicity) satisfies the asymptotics
$$
\mathcal{N}(\theta)=c_{d,N}\theta^2 d^{\lfloor \frac{\theta N-1}{d-1}\rfloor}+O(\theta \mathrm{e}^{\theta\t_c}),\quad\mbox{as}\quad \theta\to \infty,
$$
where $\lfloor x\rfloor\in \Z$ denotes the floor function of $x\in \R$, and 
$$c_{d,N}=\frac{N^3}{2(d-1)d^{2}}.$$
The proof of Theorem \ref{thm:Weyl_law} is built on us being able to explicitly compute the heat trace $\Tr(\mathrm{e}^{-t|D|})$, from which we can extract the leading terms of its Laplace inverse transform. With some further work in Lemma \ref{thm:localWeyl_law}, the proof of Theorem \ref{thm:Weyl_law} extends to a local Weyl law describing the asymptotics of $\Tr(a\chi_{[0,\theta]}(|D|))$, for $a\in \LC$.

\begin{remark}
\label{connestrace}
We note that Items (1)--(4) of Theorem \ref{thmC} can be interpreted as different incarnations of Connes' trace theorem \cite{connes88,connesredbook}. In its original form, somewhat simplified, it states that if $M$ is an $n$-dimensional compact Riemannian manifold with Laplace--Beltrami operator $\Delta$, then for any Dixmier trace $\Tr_\omega$ it holds that 
\begin{equation}
\label{ctfafdom} 
\Tr_\omega(a(1+\Delta)^{-n/2})=\frac{1}{n2^{n-1}\pi^{n/2}\Gamma(n/2)} \int_M a\, \mathrm{d}V.
\end{equation}
The reader can find more details in the textbook \cite{LSZ}. Item (3) of Theorem \ref{thmC} is a clear analogue of \eqref{ctfafdom} for Cuntz--Krieger algebras, where the role of integration is played by the KMS-state $\varphi_A$.
\end{remark}

\begin{remark}
\label{qedisc}
We note that Item (5) of Theorem \ref{thmC} has an interpretation in terms of quantum ergodicity. Classical quantum ergodicity, as surveyed in \cite{zelditchqe}, concerns the limiting behaviour of eigenfunctions of the Laplacian on a compact manifold. More precisely, if $M$ is a compact Riemannian manifold with Hodge Laplacian $\Delta$ and associated ON-basis $(f_n)_{n=1}^\infty\subseteq L^2(M)$ ordered by non-decreasing eigenvalues of $\Delta$, then the microlocal states 
$$\rho_n(a):=\langle f_n,Op(a)f_n\rangle_{L^2(M)}, \quad a\in C^\infty(S^*M),$$
converge along a density one subsequence to the Liouville measure on $S^*M$. For more details, see \cite[Theorem 1]{zelditchqe} and further references therein.

In the case at hand, we note that Item (5) of Theorem \ref{thmC} implies that if $(e_n)_{n=1}^\infty\subseteq L^2(G_A)$ is an enumeration of the eigenbasis in Subsection \ref{subsec:ham}, ordered by non-decreasing singular values $(\mu_n)_{n=1}^\infty$ of $D$, then it holds that 
$$\lim_{\theta\to \infty}\frac{1}{\mathcal{N}(\theta)}\sum_{\mu_n\leq \theta} \langle e_n, ae_n\rangle_{L^2(G_A,\mu_{G_A})}=\varphi_A(a).$$
Therefore, we arrive at an averaged form of quantum ergodicity where the sequence of states 
$$\rho_n(a)=\langle e_n, ae_n\rangle_{L^2(G_A)}, \; n=1,2,\ldots.$$
converge in Ces{\`a}ro mean along a subsequence. We prove in Proposition \ref{aldjnadljn} that $\rho_n$ does not converge along any density one subsequence, so we cannot hope for much better than Ces{\`a}ro convergence of $(\rho_n)_n$. 
\end{remark}

\begin{remark}
We note that there is difference in the orders appearing in item (3) and (4) of Theorem \ref{thmC}. Comparing the two, we have 
$$ \mathrm{Res}_{s=3}\mathrm{Tr}(a|D|^{-s}\mathrm{e}^{-\t_c |D|})=\t_c\Tr_{\omega}(a |D|^{-2}\mathrm{e}^{-\t_c |D|}).$$
For an operator with a classical Weyl law, i.e. eigenvalues behaving like $\lambda_j\sim j^\alpha$, the residue on the left hand side would be taken at the order appearing in the right hand side. 

To explain this discrepancy between the orders, we consider the model case that $D_0$ is a diagonal operator on $\ell^2(\N_{>1})$ with eigenvalues $\lambda_n=\log (cn)-2\log\log(n)$ for some $c>0$, so $X_0:=\mathrm{e}^{D_0}$ has eigenvalues $\mathrm{e}^{\lambda_n}=c\frac{n}{\log(n)^2}$. The operator $X_0$ has a more classical, yet still non-classical, Weyl law $\mathcal{N}_{X_0}(\theta)\sim\frac{\theta}{c}(\log(\theta))^2$. We have that $|D_0|^{-2}\mathrm{e}^{-D_0}$ has eigenvalues behaving like $c^{-1}n^{-1}+O\left(\frac{\log \log(n)}{n\log(n)}\right)$, so $|D_0|^{-2}\mathrm{e}^{-D_0}$ belongs to the Dixmier--Macaev ideal $\mathcal{M}_{1,\infty}$ and $\Tr_\omega(|D_0|^{-2}\mathrm{e}^{-D_0})=c^{-1}$ for any $\omega$. Note that the correct interpretation of $t_c$ in this model case is $t_c=1$.

On the other hand, 
$$\zeta(s):=\Tr(D_0^{-s}\mathrm{e}^{-D_0})=\sum_{n=2}^\infty \left(\log(cn)-2\log(\log(n))\right)^{-s}\frac{\log(n)^2}{n},$$
converges for $\mathrm{Re}(s)>3$. We can determine the pole structure from noting that $\Gamma(s)\zeta(s)$ is the Mellin transform of 
$$h(t)=\sum_{n=2}^\infty \mathrm{e}^{-(t+1)\lambda_n}=c^{-t-1}\sum_{n=2}^\infty\frac{\log(n)^{2(t+1)}}{n^{t+1}}$$ 
We see that as $t\to 0$, $h(t)=c^{-1}h_0(t)+O(t^{-2})$ where
\begin{align*}
h_0(t)=&\int_1^\infty \frac{\log(x)^{2(t+1)}}{x^{t+1}}\,\mathrm{d}x=\int_0^\infty u^{2(t+1)}\mathrm{e}^{-tu}\,\mathrm{d}u=\\
=&t^{-2t-3}\int_0^\infty z^{2(t+1)}\mathrm{e}^{-z}\,\mathrm{d}z=t^{-2t-3}\Gamma(2t+3).
\end{align*}
In the first equality we made the substitution $u=\log(x)$ and in the second equality we made the substitution $z=tu$. In particular, $h(t)=2c^{-1}t^{-3}+O(t^{-2})$ as $t\to 0$ and from standard techniques for Mellin transforms \cite[Proposition 5.1]{gs} we conclude
$$\mathrm{Res}_{s=3}\zeta(s)=c^{-1}=\Tr_\omega(|D_0|^{-2}\mathrm{e}^{-D_0}).$$ 
\end{remark}

\begin{problema}
For the reasons discussed in Remark \ref{qedisc}, we believe it is an interesting problem to prove a Weyl law similar to Theorem \ref{thm:Weyl_law}, and a local Weyl law similar to Lemma \ref{thm:localWeyl_law}, for general $A$.
\end{problema}

\begin{problema}
The next natural step after computing heat trace asymptotics and studying the leading term behaviour of the spectral geometry as in Theorem \ref{thmC}, would be to study the pole structure in heat traces $\Tr(T\mathrm{e}^{-t|D|})$. This is especially interesting for $T$ being a noncommutative differential form. A remaining open problem is to understand the JLO-cocycle \cite{jlopaper} of $(\LC,L^2(G_A),D)$ with the hope of relating it to Connes--Moscovici's local index theorem in the scaling limit $t\to \t_c^+$. Due to the fundamental role of KMS-states in the JLO-cocycle \cite{jlopaperKMS}, results such as Theorem \ref{thmC} are important first steps.
\end{problema}

\subsection*{Acknowledgements} 
The second listed author was supported by the Swedish Research Council Grant VR 2025-03923 and the Knut and Alice Wallenberg Foundation KAW 2024.0293. The authors would like to thank Dimitra Eleftheriou, Sabrina Kombrink, Bram Mesland, Adam Skalski, and Efthymios Sofos for valuable discussions.

\section{Preliminaries}

In this section we will recall the relevant background material. First, we overview topological Markov chains, and the natural metric measure space structure on the associated subshift of finite type, in Subsection \ref{topmark}. In the succeeding subsection \ref{subsec:drgroup} we recall the construction of the Deaconu--Renault groupoid describing the dynamics of a topological Markov chain. The associated $C^*$-algebra, the Cuntz--Krieger algebra, and its gauge action and KMS-state are discussed in Subsection \ref{subsec:dkal}. All of this material is combined into the spectral triple $(\LC,L^2(G_A),D)$ that we recall in Subsection \ref{subsec:ham} from \cite{GGM25}.

\subsection{Topological Markov chains}
\label{topmark}
Let $A\in M_N(\{0,1\})$ be a primitive matrix. We denote by $V_A^k$, for $k\geq 0$, the collection of admissible words of length $k$, where $V_A^0$ contains only the empty word ${\o}$, and define $$V_A:= \bigsqcup_k V_A^k.$$ The length $k$ of $\alpha\in V_A^k$ will be denoted by $|\alpha|$. 

Also, we equip the set $\{1,\ldots , N\}^{\mathbb N}$ with the product topology and consider the closed subset $$\Omega_A := \{x=(x_n)_{n\in \mathbb N} \in \{1,\ldots , N\}^{\mathbb N}: A_{x_n,x_{n+1}}=1,\, n\in \mathbb N\}.$$ A basis of open compact sets for the topology on $\Omega_A$ is given by the cylinder sets associated to finite words $\alpha = \alpha_1\ldots \alpha_n$ defined as $$C(\alpha):= \{x\in \Omega_A: x_i = \alpha_i,\,\, \text{for } 1\leq i \leq n\}.$$ The convention is that $C({\o}):= \Omega_A.$ Further, the dynamics on $\Omega_A$ is given by the left shift map $\sigma: \Omega_A \to \Omega_A$, which is a local homeomorphism. The pair $(\Omega_A,\sigma)$ is called a topological Markov chain. 

Given $\lambda >1$, we equip $\Omega_A$ with the ultrametric
$$
d(x,y):=\lambda^{-\inf \{n-1: x_n\neq y_n\}},
$$
with the convention that $\inf \varnothing = \infty$. Observe that for every $x\in \Omega_A$, the $d$-ball
$
B(x,\lambda^{-n})=C(x_1\ldots x_n).
$
Also, for $n>1$ the shift map restricts to a homeomorphism 
$B(x,\lambda^{-n})\to B(\sigma(x),\lambda^{-n+1}).$ 
The well-known fact that the topological Markov chain $(\Omega_A,\sigma)$ is a locally expanding dynamical system now follows. 

The canonical $\sigma$-invariant measure on $(\Omega_A,\sigma)$ is the so-called Parry measure. Let $\lambda_{A}>1$ denote the Perron--Frobenius eigenvalue of $A$, $u=(u_j)_{j=1}^N$ be the corresponding eigenvector and $v=(v_j)_{j=1}^N$ be the Perron--Frobenius eigenvector of $A^T$. Also, one can normalise $u$ so that $\sum_{j=1}^{N}u_j=1.$ This normalisation makes $u$ a probability distribution on $\{1,\ldots,N\}$. We also normalise $v$ so that $u\cdot v =1$. We denote this distribution by $p:=(p_j=u_jv_j)_{j=1}^{N}$. In addition, we have the stochastic matrix $P\in M_N([0,1])$,
\begin{equation}
\label{transprob}
P_{i,j}:=\frac{A_{i,j}u_j}{\lambda_{A}u_i},
\end{equation}
which measures the probability of transitioning from a vertex $i$ to $j$. Its important feature is that $p$ is stationary for $P$, i.e. $pP=p$. This gives rise to the $\sigma$-invariant Parry measure $\mu$ defined on cylinder sets $C(\alpha)$ for $\alpha=\alpha_1\ldots \alpha_n \neq {\o}$ by 
\begin{equation}
\label{eq:Parrymeas}
\mu(C(\alpha)):=p_{\alpha_1}P_{\alpha_1,\alpha_2}\ldots P_{\alpha_{n-1},\alpha_n}=\frac{v_{\alpha_1}u_{\alpha_n}}{\lambda_{A}^{n-1}}.
\end{equation}
It is well-known that $\mu$ is ergodic and its measure entropy $\ent_{\mu}$ is maximal, i.e. it coincides with the topological entropy $\log \lambda_A>0$. Moreover, there is some $C\geq 1$ so that for every $x\in \Omega_A$ and $0\leq r \leq 1$,
\begin{equation}
\label{eq:Ahlfors_reg}
C^{-1}r^{\df}\leq \mu(B(x,r))\leq Cr^{\df}, \quad\mbox{for}\quad \df=\log_{\lambda}(\lambda_A).
\end{equation}
Hence, $\mu$ is Ahlfors $\df$-regular. The exponent $\df>0$ is the Hausdorff dimension of $(\Omega_A,d).$ 

To obtain the conformal version of $\mu$, consider the continuous function $h:\Omega_{A}\to (0,\infty)$ defined as $h(x):=v_{x_1}$. Then, we can define the Borel probability measure 
$$\d \mu_c:=\frac{1}{h}\d \mu.$$ 
In particular, we have that 
$$\mu_c(C(\alpha))=\frac{u_{\alpha_n}}{\lambda_{A}^{n-1}}.$$ 
The key property of $\mu_c$ is that it is $\lambda_A$-conformal, and so for all $0\leq k\leq |\alpha|$ we have 
\begin{equation}
\label{eq:conformal}
\mu_c(\sigma^{k}(C(\alpha)))=\lambda_{A}^k\mu_c(C(\alpha)).
\end{equation}
Moreover, $\mu_c$ is a $\lambda_{A}$-eigenvector of the dual of the transfer map $L:C(\Omega_{A})\to C(\Omega_{A})$, $$\mathcal{L}f(x):=\sum_{\sigma(y)=x} f(y).$$

\subsubsection*{Notation}

\begin{enumerate}
\item For $\beta=\beta_1\ldots \beta_n \in V_A\setminus \{{\o}\}$ denote by $\hat{\beta}:=\beta_1\ldots \beta_{n-1}$ its parent. For $n=1$, we define $\beta_0={\o}.$
\item For $i\in \{1,\ldots, N\}$ set 
$$\mathfrak{C}(i):=\{j\in \{1,2,\ldots, N\}: \, A_{ij}=1\}, \quad\mbox{and}\quad \mathfrak{c}(i):=\#\mathfrak{C}(i).$$
\item By $\beta V_{A}$ we denote the set of admissible finite words starting with $\beta$.
\end{enumerate}

\subsection{Deaconu--Renault groupoid}
\label{subsec:drgroup}

The \'etale groupoid of interest is
$$
G_A=\{(x,n,y)\in \Omega_A\times \mathbb Z \times \Omega_A: \exists k\geq 0\,\, \text{such that}\,\,n+k\geq 0 \,\,\text{and}\,\,  \sigma^{n+k}(x)=\sigma^{k}(y)\}\rightrightarrows \Omega_A.
$$
The groupoid structure is given by the source and range maps 
$$
s(x,n,y):=y,\qquad r(x,n,y):=x,
$$
as well as the partial multiplication and inversion 
$$m((x,n,y),(y,\ell,z))\equiv(x,n,y)(y,\ell,z):=(x,n+\ell,z),\quad (x,n,y)^{-1}:=(y,-n,x).$$ Further, we have the maps $\kappa:G_A\to \mathbb N\cup \{0\}$ and $c:G_A\to \mathbb Z$ defined as 
\begin{align*}
\kappa(x,n,y)&:=\min \left\{k\geq \max \{0,-n\}: \sigma^{n+k}(x)=\sigma^k(y)\right\},\\
c(x,n,y)&:=n.
\end{align*}
Then, $G_A$ is equipped with the coarsest topology making the maps $s,r,\kappa,c$ continuous. Specifically, $G_{A}$ is a locally compact, totally disconnected \'etale groupoid. Moreover, the map $c$ is a groupoid homomorphism and the formula
\begin{equation}
\alpha_{z}(f)(x,n,y):=z^{n}f(x,n,y),\quad f\in C_{c}(G_{A}), \,\, z\in\mathbb{T},\end{equation}
defines an action of the circle $\mathbb{T}$ by $*$-automorphisms on the reduced $C^{*}$-algebra $C^{*}_{r}(G_{A})$, called the gauge action. A basis for the topology of $G_{A}$ is given by the bisections indexed by admissible finite words $\alpha$, $\beta$
$$C(\alpha,\beta):=\{(x,|\alpha|-|\beta|,y): x\in C(\alpha), y\in C(\beta), \sigma^{|\alpha|}(x)=\sigma^{|\beta|}(y)\}.$$

A better understanding of $G_A$ comes from the decomposition into finer bisections, namely 
\begin{equation}\label{eq:decomp}
G_A=\bigsqcup_{\gamma \in I_A} G_{\gamma}.
\end{equation}
The index set $I_A$ is defined as 
\begin{equation*}
I_A:=\{\gamma=\alpha.\beta\in V_A\times (V_A\setminus \{{\o}\}):\; \alpha={\o}\,\, \mbox{or}\,\, (\alpha\neq {\o}, A_{\alpha_{|\alpha|},\beta_{|\beta|}}=1, \alpha_{|\alpha|}\neq \beta_{|\beta|-1})\},
\end{equation*}
and the bisection $G_{\gamma}$ is given by
\begin{equation*}
G_{\gamma}:=\left\{(x,n,y)\in G_A: \begin{matrix} x=\alpha\sigma^{|\beta|-1}(y), \; y=\beta\sigma^{|\beta|}(y),\; \mbox{and}\\ 
n=|\alpha|-|\beta|+1, \; \kappa(x,n,y)=|\beta|-1\end{matrix} \right\}.
\end{equation*}
On $I_A$ we have range and source maps into $V_A$. For $\gamma=\alpha.\beta\in I_A$, we write
$$r(\gamma):=\alpha\quad\mbox{and}\quad s(\gamma):=\beta.$$ By construction,  
$\kappa|_{G_\gamma}=|s(\gamma)|-1 \quad\mbox{and}\quad c|_{G_\gamma}=|r(\gamma)|-|s(\gamma)|+1.$
Moreover, for each $\gamma=\alpha.\beta \in I_A$ the clopen set $G_\gamma$ is a subset of $C(\alpha \beta_{|\beta|},\beta)$, and hence a bisection as well. The restricted source map 
$$s_{\gamma}:=s|_{G_\gamma}:G_{\gamma}\to C(s(\gamma))$$
is a homeomorphism.

\begin{remark}\label{rem: char_functions}
For $\alpha \in V_A^1$ let $\chi_{\alpha}$ be the characteristic function of $\{ (x,1,\sigma(x)): x\in C(\alpha)\}.$ For $\alpha = {\o}$ we set $\chi_{{\o}}=1$. Then, for $\alpha=\alpha_1\ldots \alpha_n$ we write $\chi_{\alpha}:= \chi_{\alpha_1}\star \ldots \star \chi_{\alpha_n}.$ Moreover, for every $\alpha.\beta \in I_A$ and $\nu \in V_A$ such that $\beta \nu \in V_A$ we have $$\chi_{\alpha \beta_{|\beta|} \nu} \star \chi_{\beta \nu}^*=\chi_{s^{-1}_{\alpha.\beta} (C(\beta \nu))}.$$ In particular, for $\nu = {\o}$ the function $\chi_{\alpha \beta_{|\beta|}}\star \chi_{\beta}^*$ is the characteristic function of $G_{\alpha.\beta}\subset G_A.$
\end{remark}

\subsection{Cuntz--Krieger algebras}
\label{subsec:dkal}

The Cuntz--Krieger algebra $O_{A}$ is the universal $C^{*}$-algebra generated by $N$ elements $S_{i}$, $i\in\{1,\cdots, N\}$, subject to the relations
\begin{equation}\label{eq: CK-relations} 
\sum_{i=1}^{N}S_{i}S_{i}^{*}=1 \quad \mbox{and}\quad  S_{i}^{*}S_{k}=\delta_{i,k}\sum_{j=1}^{N}A_{ij}S_{j}S_{j}^{*}.
\end{equation}
The universal $C^{*}$-algebra $O_{A}$ carries an action of the group $\mathbb{T}$ by $*$-automorphisms, known as the \emph{gauge action} and determined by the formula 
\begin{equation}
\label{eq:gauge-action}
\alpha_{z}(S_{i}):=zS_{i},\quad z\in\mathbb{T},\,\,i\in\{1,\cdots, N\}.
\end{equation} 

The $C^*$-algebra $O_A$ is canonically isomorphic to $C^*_r(G_A)$ through the map 
$$S_i\mapsto \chi_i.$$ 
This isomorphism is also equivariant for the respective $\mathbb T$-actions. Under this isomorphism, the $*$-algebra $\mathbb C^* [S_i: 1\leq 1 \leq N]$ is identified with the convolution $*$-algebra $C_c^{\infty}(G_A)$.

Moreover, the canonical KMS-state $\varphi_A:O_A\to \mathbb C$ at inverse temperature $\beta=\log \lambda_A$ is given, for every $\alpha,\beta\in V_A\setminus \{{\o}\}$, by 
$$\varphi_A(S_{\alpha}S^*_{\beta})=\delta_{\alpha,\beta} \lambda_{A}^{-|\alpha|+1}u_{\alpha_{|\alpha|}}.$$ 
We let $L^2(O_A,\varphi_A)$ denote the GNS-representation of $O_A$ with respect to $\varphi_A$.

\subsection{Hamiltonians on Deaconu--Renault groupoids}
\label{subsec:ham}

We now equip $G_A$ with the conformal structure lifted from $\Omega_A$. For every $\gamma \in I_A$, we define the pull-back metric $d_{\gamma}$ on $G_{\gamma}$ by 
$$d_{\gamma}(g_1,g_2):=d(s_{\gamma}(g_1),s_{\gamma}(g_2)).$$ 
Hence, the restricted source map $s_\gamma:G_\gamma\to C(s(\gamma))$ is an isometry. Then, we define an extended metric $d_{G_{A}}$ on $G_A$, generating the topology on $G_A$, by 
$$d_{G_A}(g_1,g_2):=
\begin{cases}
d_{\gamma}(g_1,g_2), & \text{if  for some $\gamma\in I_A$ both}\; g_1,g_2\in G_{\gamma}\\
\infty, & \text{otherwise}
\end{cases}.
$$

We define the pull-back finite Borel measure $\mu_{\gamma}:=s_\gamma^*\mu_c$ on $G_{\gamma}$. For an open $B\subset G_{\gamma}$, $\mu_\gamma$ is defined by 
$$\mu_{\gamma}(B)= \mu_c(s_{\gamma}(B)).$$ 
The collection $\{\mu_{\gamma}\}_{\gamma\in I_A}$ defines a Borel measure $\mu_{G_A}$ on $G_A$. Now since each $s_{\gamma}(G_{\gamma})$ is clopen in $\Omega_A$, and for the latter $\mu_c$ is Ahlfors $\df$-regular, we obtain that for every $\gamma \in I_A$, the metric-measure space $(G_{\gamma},d_{\gamma},\mu_{\gamma})$ is Ahlfors $\df$-regular with $\df= \log_{\lambda} \lambda_A.$

Further, we have that 
$$L^2(G_A, \mu_{G_A})=\bigoplus_{\gamma\in I_A}L^2(G_{\gamma},\mu_{\gamma}).$$ 
Therefore, we can define the positive, self-adjoint operator $\Delta$ acting on $L^2(G_A, \mu_{G_A})$ as 
$$\Delta=\bigoplus_{\gamma\in I_A}\Delta_{\gamma},$$ 
where each $\Delta_{\gamma}$ is the log-Laplacian on $L^2(G_{\gamma},\mu_{\gamma})$. We have $\ker \Delta_{\gamma} = \mathbb C \chi_{G_{\gamma}}$ and $\Delta_{\gamma}$ is essentially self-adjoint on locally constant functions $f:G_{\gamma}\to \mathbb C$ for which $$\Delta_{\gamma} f(g) = \int_{G_{\gamma}} \frac{f(g)-f(h)}{d_{\gamma}(g,h)^{\df}}\d \mu_{\gamma} (h).$$ From the definition of $d_{\gamma}$ and $\mu_{\gamma}$, if $\gamma=\alpha.\beta$ then $\Delta_{\gamma}=\Delta_{C(\beta)}$, where 
$$\Delta_{C(\beta)}f(x)=\int_{C(\beta)} \frac{f(x)-f(y)}{d (x,y)^{\df}}\rd\mu_c(y).$$ 
Also, the space $\LC$ is a core for $\Delta$. Observe that for $f\in \LC$ the operator $\Delta$ admits the integral representation 
\begin{equation}\label{eq:Int_rep}
\Delta f(g)=\int_{G_A} \frac{f(g)-f(h)}{d_{G_A}(g,h)^{\df}} \d \mu_{G_A}(h).
\end{equation}

The measure $\mu_{G_A}$ relates to the KMS-state $\varphi_A$ via the fact that, for $f\in C_c(G_A)$, we have 
\begin{equation}\label{eq:KMS}
\varphi_A(f)=\int_{\Omega_A} f(x,0,x) \d \mu_{G_A} (x).
\end{equation}
This identity identifies $L^2(O_A,\varphi_A)$ with $L^2(G_A,\mu_{G_A})$ under the isomorphism $O_A\cong C^*_r(G_A)$. 

Further, $\Delta$ admits a spectral decomposition with eigenfunctions forming an ON-basis $(\mathrm{e}_{\gamma},\mathrm{e}_{(\gamma,\nu,j)})_{(\gamma,\nu,j)\in \mathfrak{I}_A}\subset \LC$  for $L^2(G_A,\mu_{G_A})$. Here $\mathfrak{I}_A$ is an appropriate index set. Namely, 
$$\mathrm{e}_{\gamma}:=\mu_{\gamma}(G_{\gamma})^{-1/2} \chi_{G_{\gamma}}.$$ 
Also, each $\mathrm{e}_{(\gamma,\nu,j)}$ is a Haar wavelet supported on $s^{-1}_{\gamma}(C(\nu))$, with $\nu \in V_A \setminus \{{\o}\}$ starting with $s(\gamma)\neq {\o},$ and $j$ lies in $J_{s(\gamma)}(\nu)\subseteq \{1,\ldots, N\}$; a non-empty index set such that 
$$\# J_{s(\gamma)}(\nu)= \# \{i\in \{1,\ldots, N\}:A_{\nu_{|\nu|},i}=1\}-1\geq 1.$$ 
Although in \cite{GGM25} these eigenfunctions are described algorithmically with explicit formulas only in special cases, in Proposition \ref{prop:eigenfunctions} of the current paper we obtain explicit formulas for general $A$. 

Moreover, we have that $\ker(\Delta)$ is the closed linear span of $(\mathrm{e}_{\gamma})_{\gamma\in I_A}$ and for $(\gamma,\nu,j)\in \mathfrak{I}_A$, 
$$\Delta \mathrm{e}_{(\gamma,\nu,j)}=\lambda_{s(\gamma)}^A(\nu)\mathrm{e}_{(\gamma,\nu,j)},$$ where $\lambda^A_{s(\gamma)}(s(\gamma))=\lambda_{A}u_{s(\gamma)_{|s(\gamma)|}}$, and if $\nu\in s(\gamma)V_A \setminus \{s(\gamma)\}$ then
\begin{equation}
\label{eq:eigenvalues}
\lambda_{s(\gamma)}^A(\nu)=\lambda_{A}u_{\nu_{|\nu|}}+\lambda_{A}\sum_{k=0}^{|\nu|-|s(\gamma)|-1} u_{\nu_{|s(\gamma)|+k}}(1-P_{\nu_{|s(\gamma)|+k},\nu_{|s(\gamma)|+k+1}}).
\end{equation}

For the sequel it will be important to recall how $O_A$ acts on the eigenfunctions of $\Delta$. 

\begin{thm}[{\cite[Theorem 3.9]{GGM25}}]
\label{thm:O_A_action}
Let arbitrary $i\in \{1,\ldots ,N\}$, $\gamma=\alpha.\beta \in I_A$ and ${e}_{(\gamma,\nu,j)}$. If 
\begin{enumerate}
\item $\alpha \neq {\o}$ or $|\beta|=1$, and 
\vspace{0.2cm}
\begin{enumerate}
\item $i\alpha.\beta \not\in I_A$, then $S_i\mathrm{e}_{(\gamma,\nu,j)}=0$ and $S_i\mathrm{e}_{\gamma}=0$;
\item $i\alpha.\beta \in I_A$, then $S_i\mathrm{e}_{(\gamma,\nu,j)}=\mathrm{e}_{(i\gamma,\nu,j)}$ and $S_i\mathrm{e}_{\gamma}=\mathrm{e}_{i\gamma}.$
\end{enumerate}
\vspace{0.2cm}

\item $\alpha={\o}, |\beta|\geq 2$, $\beta_{|\beta|-1}\neq i$, and
\vspace{0.2cm}

\begin{enumerate}
\item $i.\beta \not\in I_A$, then $S_i\mathrm{e}_{(\gamma,\nu,j)}=0$ and $S_i\mathrm{e}_{\gamma}=0$;
\item $i.\beta \in I_A$, then $S_i\mathrm{e}_{(\gamma,\nu,j)}=\mathrm{e}_{(i\gamma,\nu,j)}$ and $S_i\mathrm{e}_{\gamma}=\mathrm{e}_{i\gamma}$.
\end{enumerate}
\vspace{0.2cm}

\item $\alpha={\o}, |\beta|\geq 2$, $\beta_{|\beta|-1}=i$, 
\vspace{0.2cm}
then $$S_i\mathrm{e}_{(\gamma,\nu,j)}=\mathrm{e}_{(\hat{\gamma},\nu,j)},\qquad S_i\mathrm{e}_{\gamma}=\mu_c(C(\beta))^{-1/2}\chi_{s^{-1}_{\hat{\gamma}}(C(\beta))}.$$ Here $\hat{\gamma}:={\o}.\hat{\beta},$ and if $C(\hat{\beta})=C(\beta)$ then $S_i\mathrm{e}_{\gamma}=\mathrm{e}_{\hat{\gamma}}.$
\end{enumerate}
\end{thm}

Consider the multiplication operator $M_{\widetilde{L}}$ by the proper continuous function $\widetilde{L}:G_A\to \mathbb N$ defined as
\begin{equation}\label{eq:scale}
\widetilde{L}|_{G_\gamma}=|\gamma|:=|r(\gamma)|+|s(\gamma)|.
\end{equation}
Note that $M_{\widetilde{L}}$ is positive and essentially self-adjoint on ${\rm C_{c}^\infty}(G_A)$.

Further, we introduce the space of finite words with a distinguished ending as 
\begin{equation}\label{eq:Fock}
\tilde{V}_A:=\{\alpha.\beta\in V_A\times \{1,\ldots,N\}: \alpha\beta\in V_A\}=\{\gamma\in I_A: |s(\gamma)|=1\},
\end{equation}
and write $\ell^2(\tilde{V}_A)$ for the associated Hilbert space.
Furthermore, consider the Hilbert space 
$$\mathcal{F}_A:=\bigoplus_{\gamma\in \tilde{V}_A}\ker \Delta_{\gamma}\subset L^2(G_A,\mu_{G_A}),$$ 
and the associated projection $P_A$ from $L^2(G_A,\mu_{G_A})$ onto $\mathcal{F}_A$.

The Hamiltonian is the operator $$\widetilde{D}:= - \Delta + \widetilde{V},\qquad \widetilde{V}:= (2P_A-1)M_{\widetilde{L}}.$$ For $f\in \LC$ we have 
\begin{equation}\label{eq:D}
\widetilde{D}f=
\begin{cases}
-(\Delta+M_{\widetilde{L}})f, &\text{if } f\in \mathcal{F}_A^{\perp}\\
M_{\widetilde{L}}f, &\text{if } f\in \mathcal{F}_A
\end{cases}.
\end{equation}
Also, it is easy to see that $\widetilde{D}=(2P_A-1)(\Delta + M_{\widetilde{L}})$ and hence $|\widetilde{D}|=\Delta + M_{\widetilde{L}}.$ Moreover,

$$\widetilde{D}\mathrm{e}_{\gamma}=
\begin{cases}
-|\gamma|\mathrm{e}_{\gamma}, &\text{if } |s(\gamma)|>1,\\
|\gamma|\mathrm{e}_{\gamma}, &\text{if } |s(\gamma)|=1
\end{cases}.
$$
and 
$$\widetilde{D}\mathrm{e}_{(\gamma,\nu,j)}=-(|\gamma|+\lambda_{s(\gamma)}^A(\nu))\mathrm{e}_{(\gamma,\nu,j)}.$$ We have that $(\LC,L^2(G_A, \mu_{G_A}),\widetilde{D})$ is a spectral triple over $O_A$, whose K-homology class is the image of $[1]\in K_0(O_{A^T})$ under KK-duality. 

\begin{remark}\label{rem:normalise_Ham}
The definition of $\widetilde{D}$ has some flexibility. Namely, given any $\eta >0$ we can define $$\widetilde{D}_{\eta}:=-\Delta + \eta(2P_A-1)M_{\widetilde{L}}.$$ Each $\widetilde{D}_{\eta}$ yields a spectral triple on $O_A$ with the same K-homology class and classical isometry group as that of $\widetilde{D}_1=\widetilde{D}$. Also, the proof of \cite{FGM} carries over to $\widetilde{D}_{\eta}$ and the quantum isometry groups $QISO(\widetilde{D}_{\eta})$ and $QISO(\widetilde{D}_{1})$ coincide, except for finitely many $\eta$ for which we can only show that $QISO(\widetilde{D}_{1})$ is a subgroup of $QISO(\widetilde{D}_{\eta})$. At this point, it is not clear to us though if this is a limitation of the proof in \cite{FGM}, or a quantisation feature. Nevertheless, in the sequel (see Definition \ref{defn:crit_time_Ham}) we will see that the spectral geometry of $O_A$ is richer for a specific $\eta>0$ associated with the thermodynamic formalism of $(\Omega_A,\sigma)$. 
\end{remark}

\section{Statistical properties of the spectral data}
\label{secstatarad}

We will now describe the non-zero eigenvalues $(\lambda_\beta^A(\nu))_{\nu \in \beta V_A}$ of $\Delta_{C(\beta)}$, for $\beta\in V_A\setminus \{\o\}$. As seen from the formula \eqref{eq:eigenvalues}, the behaviour of the eigenvalues is somewhat erratic. Nevertheless, we show, somewhat surprisingly, that if we sample $\nu$ according to an ergodic and $\sigma$-invariant measure on $\Omega_A$,  then $\lambda_\beta^A(\nu)$ will asymptotically behave like a constant multiple of $|\nu|$, where the constant multiple depends on entropy and a potential.

Denote by $\mathcal{M}_{\sigma}$ the space of $\sigma$-invariant Borel probability measures on $\Omega_{A},$ and by $\mathcal{M}_{\sigma,e}\subset \mathcal{M}_{\sigma}$ the ergodic ones. Recall the definition of the transition probabilities $P_{ij}$ from Equation \eqref{transprob}. We define the potential function $F_A:\Omega_{A}\to [0,1)$ given by 
$$F_A(x):=u_{x_1}(1-P_{x_{1}, x_{2}}).$$

\begin{thm}
\label{lem:eigen_asym}
For $n\in \mathbb N$ consider the projection onto the first $n$ terms $\pi_n:\Omega_{A}\to V_{A}^n$. We fix $\tau \in \mathcal{M}_{\sigma,e}$. Then, for $\tau$-almost every $x\in C(\beta)$ we have $$\lim_{n\to \infty} \frac{\lambda_\beta^{A}(\pi_n(x))}{n}=\lambda_{A} \tau(F_A).$$ Moreover, there is some $0<c_0\leq c_1<1$ so that $c_0\leq \tau(F_A)\leq c_1 \lambda_A^{-1}$. 
\end{thm}

\begin{proof}
Note that $F_A$ is locally constant, hence continuous. From the ergodic theorem we then get that for $\tau$-almost every $x\in \Omega_{A}$, $$\lim_{n\to \infty}\frac{1}{n}\sum_{k=0}^{n-1}F_A(\sigma^k(x))=\tau(F_A).$$ In particular, this holds for $\tau$-almost every $x\in C(\beta)$ for which, if $n>|\beta|$, then 
\begin{align*}
\sum_{k=0}^{|\pi_n(x)|-|\beta|-1} u_{\pi_n(x)_{|\beta|+k}}(1-P_{\pi_n(x)_{|\beta|+k},\pi_n(x)_{|\beta|+k+1}})&= \sum_{k=|\beta|-1}^{n-2} u_{x_{k+1}}(1-P_{x_{k+1},x_{k+2}})\\
&= \sum_{k=|\beta|-1}^{n-2} F_A(\sigma^k(x)).
\end{align*}
Since $\lambda_{A}u_{\nu_{|\nu|}}$ is uniformly bounded for $\nu\in \beta V_{A}$, $F_A$ is bounded, and $\beta$ is fixed, we have $$\lim_{n\to \infty} \frac{\lambda_\beta^{A}(\pi_n(x))}{n}=\lambda_{A}\tau(F_A).$$ 

Further, it is not difficult to see that $\tau(F_A)\neq 0$ using the primitivity of the matrix $A$. However, the fact that $\tau(F_A)\geq c_0$ for some constant $c_0>0$ independent of $\tau$ is derived from \cite[Proposition 4.6]{GGM25}.

Finally, due to the $\lambda_A$-conformality of $\mu_c$ (see \eqref{eq:conformal}) for all $i\in \{1,\ldots, N\}$ we have $$u_i\leq \lambda_A^{-1}.$$ The proof is complete by setting $c_1:=\max \{1-P_{i,j}: A_{i,j}=1\}$. 
\end{proof}

For understanding the potential $F_A$ and the asymptotic behaviour of the eigenvalues recall the distributions $u,p$ and the stochastic matrix $P$ from Subsection \ref{topmark}. Then, consider the product space $V_{A}^1\times V_{A}^1$ with stochastic matrix $$(P\otimes P)_{(i,i'),(j,j')}:=P_{i,j}P_{i',j'},$$ as well as the distribution $(p\otimes u)_{(i,i')}:=p_i u_{i'}$. This is stationary for $P\otimes P$ if and only if $u$ is stationary for $P$. In turn, this is equivalent to the in-degree (sum of columns in $A$) of every vertex to be $\lambda_{A}$. In general, if the columns of $A$ sum to the same constant $c$, then $A^T1=c1$ and from Perron--Frobenius theorem we have $c=\lambda_{A}.$ In that case we can normalise so that $v=1$, and hence $p=u$. Nevertheless, we have exponential convergence $uP^n\to p$, and similarly $$(p\otimes u) (P\otimes P)^n\to p\otimes p.$$ Let $(X_n)_{n\geq 0}$ be an one-step Markov chain on $V_{A}^1\times V_{A}^1$ with $\text{Law}(X_n)=p\otimes uP^n.$ Then, considering the diagonal $\mathcal{D}\subset V_{A}^1\times V_{A}^1$ we are interested in the following quantity 
$$s_{A}(n):=\Prob(X_{n}\in \mathcal{D}, X_{n+1}\not\in \mathcal{D}).$$ 
This is independent of $n\geq 0$ if $u$ is stationary, and in fact it easily follows that in general
$$s_{A}(n)=\sum_{i=1}^{N} \left[\sum_{m=1}^{N}(1-P_{i,m})P_{i,m} \right] p_i (uP^n)_i.$$ 
The quantity $\sum_{m=1}^{N}(1-P_{i,m})P_{i,m}$ is the so-called Gini impurity at the vertex $i$.

\begin{prop}\label{prop:sA}
It holds that $\mu(F_A)=s_A(0)$, where $\mu$ is the Parry measure \eqref{eq:Parrymeas}.
\end{prop}

\begin{proof}
We have that
\begin{align*}
\int_{\Omega_{A}}F_A(z)\d \mu(z)&= \sum_{i,m=1}^{N}A_{i,m}\int_{C(im)}F_A(z) \d \mu(z)\\
&=\sum_{i,m=1}^{N}A_{i,m}u_i(1-P_{i,m})\mu (C(im))\\
&=\sum_{i=1}^{N} \left[\sum_{m=1}^{N}(1-P_{i,m})P_{i,m} \right] p_i u_i\\
&=s_{A}(0).
\end{align*}
\end{proof}

\begin{prop}\label{prop:sA2}
The quantity $\tau(F_A)$ is independent of $\tau\in \mathcal{M}_{\sigma,e}$ if and only if there is a uniform constant $c>0$ such that, for every $n\in \mathbb N$ and $\gamma\in V_A^n$ with $A_{\gamma_n,\gamma_1}=1$, it holds 
\begin{equation}\label{eq:sA2_0}
\frac{1}{n}\sum_{i=1}^{n} u_{\gamma_i} =c.
\end{equation}
Moreover, \eqref{eq:sA2_0} can happen only for $c=\sum_{i=1}^{N} p_i u_i.$ Further, $F_A$ is constant if and only if all row-sums of $A$ are equal (the graph is out-regular).
\end{prop}

\begin{proof}
Assume that $\tau(F_A)$ is independent of $\tau\in \mathcal{M}_{\sigma,e}$ and set $c':=\tau(F_A)$. Then, clearly 
\begin{equation}\label{eq:sA2_1}
\tau(c')=\tau(F_A).
\end{equation}
Since $(\Omega_{A},\sigma)$ is mixing, it has specification. Therefore, $\mathcal{M}_{\sigma,e}$ is weak$^*$ dense in $\mathcal{M}_{\sigma}$, meaning that \eqref{eq:sA2_1} holds for every $\tau\in \mathcal{M}_{\sigma}$. The Liv{\v{s}}ic theorem \cite[Theorem 1.1]{Sar}  then implies that $F_A$ is cohomologous to $c'$, i.e. there is $h\in C(\Omega_{A})$ so that $$F_A=c'+h\circ \sigma - h.$$ In our case, for every $x\in \Omega_{A},$ we obtain $u_{x_1}=\lambda_{A}^{-1}u_{x_2}+c'+h(\sigma(x)) - h(x),$ and iteration gives 
\begin{equation}\label{eq:sA2_2}
u_{x_1}=c+H(x),
\end{equation} 
where $c:=c'\sum_{n\geq 0} \lambda_{A}^{-n}=c'\lambda_{A}(\lambda_{A}-1)^{-1},\, H(x):=\sum_{n\geq 0} \left(h(\sigma^{n+1}(x)) - h(\sigma^n(x))\right)\lambda_{A}^{-n}.$ It is easy to see that for $x\in \Omega_{A}$ with period $n\in \mathbb N$ it holds $$\sum_{k=0}^{n-1}H(\sigma^k(x))=0.$$ As a result, one has $$\sum_{k=0}^{n-1}u_{x_{k+1}}=nc.$$ 

For the other direction assume that \eqref{eq:sA2_0} holds and define $G:=F_A-c(\lambda_{A}-1)\lambda_{A}^{-1}$. Then, for every $x\in \Omega_{A}$ with period $n\in \mathbb N$ we have $$\sum_{k=0}^{n-1}G(\sigma^{k}(x))=0.$$ From the Liv{\v{s}}ic theorem we then obtain that $G$ is a coboundary. Hence, $F_A$ is cohomologous to a constant and thus $\tau(F_A)$ is independent of $\tau\in \mathcal{M}_{\sigma,e}$. 

Now whenever \eqref{eq:sA2_0} holds, from \eqref{eq:sA2_2} we have that $H$ is a coboundary, and since $\mu$ is $\sigma$-invariant, we obtain that $c=\mu(g)$ where $g\in C(\Omega_{A})$ is given by $g(x):=u_{x_1}.$ Clearly, it holds $\mu(g)=\sum_{i=1}^{N}p_iu_i.$

Finally, assuming $F_A$ is constant, from \eqref{eq:sA2_2} we get that $u$ is constant, meaning the row-sums of $A$ are equal. Conversely, if the row-sums are equal, the vector $(1,\ldots,1)$ is a $\lambda_{A}$-eigenvector and up to a scalar is equal to $u$. Hence, $u$ is constant and thus $F_A$ is constant.
\end{proof}

At this point we can consider the following normalisation of the Hamiltonian $\widetilde{D}$, see Remark \ref{rem:normalise_Ham}. The terminology is justified from the results in Subsection \ref{sec:sec5}, particularly Remark \ref{rem:synch}. 

\begin{definition}\label{defn:crit_time_Ham}
The quantity $\t_c\in (0,\infty)$ defined as 
$$\t_c:=\lambda_{A}^{-1}\sup\limits_{\tau\in \mathcal{M}_{\sigma,e}}\frac{\ent_{\tau}}{\tau(F_A)}$$ 
will be called the \textit{critical time} of the normalised Hamiltonian 
$$D:=-\Delta + (2P_A-1)M_L,$$ 
where the multiplication operator $M_L$ is associated with $L:G_A\to (0,\infty)$ which is defined for $\gamma\in I_A$ as 
$$L|_{G_{\gamma}}=|\gamma|_c:=\frac{\log \lambda_A}{\t_c}(|r(\gamma)|+|s(\gamma)|).$$
\end{definition}

Note that due to Theorem \ref{lem:eigen_asym}, we always have that $\t_c>\log \lambda_A$.

\begin{remark}\label{rem:crit}
Assume the row-sums of $A$ are equal to $d\geq 2$. Then, from Proposition \ref{prop:sA2}, the fact that $\lambda_A=d$, and the Variational Principle for $(\Omega_A,\sigma)$ (see \cite[Section 4.2]{Sar}), it is clear that $\t_c=N(d-1)^{-1}\log d$ and therefore
$$L|_{G_{\gamma}}=|\gamma|_c=\frac{d-1}{N}(|r(\gamma)|+|s(\gamma)|).$$ 
\end{remark}

\section{Finer structure in the heat kernel}

In this section we begin our study of the heat kernel $\mathrm{e}^{-t|D|}$. First, in Subsection \ref{subsec:eigen} we find closed formulas for the eigenfunctions of $D$, and in Subsection \ref{subsec:kernel} we put that together into an explicit description of the heat kernel for out-regular graphs.

\subsection{A basis of eigenfunctions}
\label{subsec:eigen}
Let $\beta\in V_A\setminus \{\o\}$. Turning now our attention to eigenfunctions, note that $0\in  \mathrm{Spec}(\Delta_{C(\beta)})$ has multiplicity $1$ with normalised eigenfunction 
$$h_\beta:=\mu_c(C(\beta))^{-1/2}\chi_{C(\beta)}.$$ Denote by $\overline{\beta V_{A}}$ the set of $\nu\in \beta V_{A}$ with $\mathfrak{c}(\nu_{|\nu|})\geq 2$. Following \cite{GGM25}, each $\lambda_\beta^{A}(\nu)$-eigenspace (where $\nu\in \overline{\beta V_{A}}$) has an abstract orthonormal basis $(h_{\nu ,j})_{j\in J_\beta(\nu)}$ for a certain finite index set $J_\beta(\nu)$. Here we present an explicit formula for that basis. We choose the index set for $\nu\in \overline{\beta V_{A}}$ as
$$J_\beta(\nu):=\{1,\ldots, \mathfrak{c}(\nu_{|\nu|})-1\}.$$
The orthonormal basis for $(\C h_\beta)^\perp\subseteq L^2(C(\beta),\mu_c)$ of eigenfunctions $(h_{\nu,j})_{\nu\in \overline{\beta V_{A}}, j\in J_\beta(\nu)}$ can be constructed following dyadic techniques analogous in spirit to \cite{KLPW}.

\begin{prop}\label{prop:eigenfunctions}
Let $\nu\in \overline{\beta V_{A}}$ and choose an ordering $\iota_\nu:\{1,\ldots, \mathfrak{c}({\nu_{|\nu|}})\} \to \mathfrak{C}(\nu_{|\nu|}).$ Also, for every $1\leq m\leq \mathfrak{c}({\nu_{|\nu|}})$ consider the cumulative distribution $$q_m:=\sum_{k=1}^{m} P_{\nu_{|\nu|},\iota_{\nu}(k)}.$$ Then, for every $j\in J_\beta(\nu)$ define $a^{(\nu,j)}\in \mathbb R^{\mathfrak{c}(\nu_{|\nu|})}$ as 
$$a^{(\nu,j)}_k=
\begin{cases}
\left(\frac{P_{\nu_{|\nu|},\iota_{\nu}(j+1)}}{\mu_c(C(\nu))q_j q_{j+1}}\right)^{1/2}, & 1 \leq k\leq j\\
-\left(\frac{q_j}{\mu_c(C(\nu))q_{j+1}P_{\nu_{|\nu|},\iota_{\nu}(j+1)}}\right)^{1/2}, & k=j+1\\
0, & k\geq j+2
\end{cases},
$$
and the function $h_{\nu,j}:\Omega_{A}\to \mathbb R$ as
$$h_{\nu,j}:=\sum_{k=1}^{\mathfrak{c}(\nu_{|\nu|})}a^{(\nu,j)}_k \chi_{C(\nu\iota_\nu(k))}.$$ Then, we have that
\begin{enumerate}
\item $h_{\nu,j}$ is supported on $C(\nu)$ and is constant on its children; 
\item $\int_{\Omega_{A}}h_{\nu,j}\d \mu_c=0$;
\item $\langle h_{\nu,j},h_{\nu,j'}\rangle_{L^2} = \delta_{j,j'}$;
\item $\{ \mu_c(C(\nu))^{-1/2}\chi_{C(\nu)}\}\cup \{h_{\nu,j}: j\in J_{\beta}(\nu)\}$ is an orthonormal basis for the subspace of all functions on $C(\nu)$ that are constant on its children.
\end{enumerate}
\end{prop}

\begin{proof}
We will build an orthonormal basis for the subspace of all functions on $C(\nu)$ that are constant on its children. Any such function $f$ is associated with a vector $b\in \mathbb C^{\mathfrak{c}(\nu_{|\nu|})}$ so that $$f=\sum_{k=1}^{\mathfrak{c}(\nu_{|\nu|})} b_k \chi_{C(\nu\iota_\nu(k))}.$$ Here, given $j\in J_{\beta}(\nu)$, we shall focus on the vector $$a^{(\nu,j)}:=(\underbrace{\eta_{j},\ldots ,\eta_{j}}_{j-\text{times}},\theta_{j}, 0,\ldots, 0)\in \mathbb R^{\mathfrak{c}(\nu_{|\nu|})},$$ searching for suitable $\eta_{j},\theta_{j}\in \mathbb R$. Also, for brevity set $r_k:=P_{\nu_{|\nu|},\iota_{\nu}(k)}$, for $1\leq k\leq \mathfrak{c}(\nu_{|\nu|}),$ and $m_{\nu}:=\mu_c(C(\nu)).$ 

Asking for the function associated with $a^{(\nu,j)}$ to have zero mean gives
\begin{equation}\label{eq:Wav1}
0=\sum_{k=1}^{\mathfrak{c}(\nu_{|\nu|})}a^{(\nu,j)}_k r_k
=\sum_{k=1}^{j}a^{(\nu,j)}_k r_k + a^{(\nu,j)}_{j+1}r_{j+1}
=\eta_j q_j+\theta_j r_{j+1},
\end{equation}
hence 
\begin{equation}\label{eq:Wav2}
\theta_j=-\frac{\eta_j q_j}{r_{j+1}}.
\end{equation}
Further, asking for that function to have $L^2$-norm equal to $1$ gives
\begin{align*}
1&=m_{\nu}\sum_{k=1}^{\mathfrak{c}(\nu_{|\nu|})}(a^{(\nu,j)}_k)^2 r_k\\
&=m_{\nu}\left( \sum_{k=1}^{j}(a^{(\nu,j)}_k)^2 r_k + (a^{(\nu,j)}_{j+1})^2r_{j+1}\right)\\
&= m_{\nu}\left(\eta_j^2 q_j + \theta_j^2 r_{j+1}\right)\\
&=m_{\nu}(\eta_j^2 q_j + \eta_j^2q_j^2 r_{j+1}^{-1})\\
&=m_{\nu}\eta_j^2q_j(r_{j+1}+q_j)r_{j+1}^{-1}\\
&=m_{\nu}\eta_j^2q_jq_{j+1}r_{j+1}^{-1}.
\end{align*}
Therefore, 
\begin{equation}\label{eq:Wav3}
\eta_j= \left(\frac{r_{j+1}}{m_{\nu}q_jq_{j+1}}\right)^{1/2}.
\end{equation}
Substituting \eqref{eq:Wav3} to \eqref{eq:Wav2} gives $$\theta_j=-\left(\frac{q_j}{m_{\nu}q_{j+1}r_{j+1}}\right)^{1/2}.$$ Define $h_{\nu,j}$ as the function associated with $a^{(\nu,j)}.$

Regarding orthogonality, let $j, j' \in J_{\beta}(\nu)$ and assume $j + 1 \leq j'$. Then, from \eqref{eq:Wav1} we obtain
\begin{align*}
\langle h_{\nu,j}, h_{\nu,j'}\rangle_{L^2}&=m_{\nu}\sum_{k=1}^{\mathfrak{c}(\nu_{|\nu|})} a^{(\nu,j)}_k a^{(\nu,j')}_k r_k \\
&=m_{\nu}\sum_{k=1}^{j+1} a^{(\nu,j)}_k a^{(\nu,j')}_k r_k\\
&=m_{\nu}\eta_{j'}\sum_{k=1}^{j+1} a^{(\nu,j)}_k r_k\\
&=m_{\nu}\eta_{j'}\sum_{k=1}^{\mathfrak{c}(\nu_{|\nu|})} a^{(\nu,j)}_k r_k\\&=0.
\end{align*}
This completes the proof as part (4) is trivial.
\end{proof}

\subsection{Heat kernels for out-regular graphs}
\label{subsec:kernel}

We can now write an explicit formula for the heat kernel of the log-Laplacian when the underlying graph is out-regular. The proof follows that of \cite[Proposition 4.8]{GGM25} for full shifts, but now uses the explicit form of the eigenfunctions in \ref{prop:eigenfunctions}. Also, recall the critical time $\t_c>0$ from Definition \ref{defn:crit_time_Ham}.

\begin{thm}
\label{thm:Heat_kernel_formula}
Assume that all row-sums of $A$ are equal to $d\geq 2.$ Then, for every $\gamma \in I_A$ and $t>0$, the heat operator $\mathrm{e}^{-t\Delta_{\gamma}}$ admits a kernel $k_{\gamma,t}\in L^1(G_{\gamma}\times G_{\gamma},\mu_{\gamma}\times \mu_{\gamma})$. Specifically, setting $c=(d-1)N^{-1}$, for $g_1\neq g_2 \in G_{\gamma}$ it is given by 
$$k_{\gamma,t}(g_1,g_2):=h_\gamma(t)+H_\gamma(t)\mathrm{d}_\gamma(g_1,g_2)^{-\df+\delta t},$$
where $\delta=c\log_{\lambda}(\mathrm{e}),\,\, \df=\log_{\lambda} d,$ and
\begin{align*}
h_\gamma(t)&:=d^{|s(\gamma)|-1}N\left(1-\frac{(d-1)\mathrm{e}^{-dN^{-1}t}}{d\mathrm{e}^{-ct}-1}\right)\\
H_\gamma(t)&:=\frac{N-N\mathrm{e}^{-ct}}{d\mathrm{e}^{-ct}-1}\mathrm{e}^{-t(dN^{-1}-c|s(\gamma)|)}.
\end{align*}
Further, $k_{\gamma,t}(g,g)<\infty$ if and only if $t>\t_c=\df\delta^{-1}$, and for $t>\t_c$ we have that $k_{\gamma,t}$ is continuous. Finally, 
\begin{equation}
\label{etdeltaon}
\mathrm{e}^{-t\Delta_\gamma}=d^{-|s(\gamma)|+1}N^{-1}h_\gamma(t)P_{\Delta_\gamma}+H_\gamma(t)R_{\gamma,\delta t},
\end{equation}
where $P_{\Delta_\gamma}$ denotes the projection onto the kernel of $\Delta_\gamma$ and $R_{\gamma,\delta t}$ is the Riesz potential acting on $L^2(G_{\gamma},\mu_{\gamma})$ by $$R_{\gamma,\delta t}f(g_1)=\int_{G_{\gamma}}\frac{f(g_2)}{d_{\gamma}(g_1,g_2)^{\df-\delta t}} \d \mu_{\gamma}(g_2).$$
\end{thm}

\begin{proof}
Let $\gamma \in I_A,\,\, t>0$ and observe that we can equivalently work on $C(\beta)$ for $\beta=s(\gamma)$. Denote the associated kernel by $k_{\beta,t}$, which is written in terms of the orthonormal basis $(h_{\beta},h_{\nu,j})_{\nu \in \overline{\beta V_A}, j\in J_{\beta}(\nu)}$ of $L^2(C(\beta),\mu_c)$ of Proposition \ref{prop:eigenfunctions}. Further, since the row-sums of $A$ are equal to $d$, the eigenvalue associated to each such $h_{\nu,j}$ is $\lambda_{\beta}^A(\nu)= dN^{-1}+c(|\nu|-|\beta|).$ Also, recall that the basis consists of real-valued functions.

Let $x\neq y \in C(\beta)$ and $k:=-\log_{\lambda} d(x,y)\geq |\beta|$. Also, let $w\in \overline{\beta V_A}$ ($=\beta V_A$ as the graph is out-regular) such that $|w|=k,\,\,x=wx',\,\,y=wy'$ and $x_1'\neq y_1'$. Also, by $\nu < w$ we will mean $\nu=w_1\ldots w_l$, for some $l<k$. Then, 
\footnotesize
\begin{align*}
k_{\beta,t}(x,y)&=h_{\beta}(x)h_{\beta}(y)+\sum_{\nu \in \beta V_A} \mathrm{e}^{-\lambda_{\beta}^A(\nu)t}\sum_{j\in J_{\beta}(\nu)} h_{\nu,j}(x)h_{\nu,j}(y)\\ 
&=d^{|\beta|-1}N+\sum_{\nu \in \beta V_A, \nu<w} \mathrm{e}^{-\lambda_{\beta}^A(\nu)t}\sum_{j\in J_{\beta}(\nu)} h_{\nu,j}(x)h_{\nu,j}(y)+\mathrm{e}^{-\lambda_{\beta}^A(w)t}\sum_{j\in J_{\beta}(w)} h_{w,j}(x)h_{w,j}(y)\\
&=d^{|\beta|-1}N+\sum_{\nu \in \beta V_A, \nu<w} \mathrm{e}^{-\lambda_{\beta}^A(\nu)t}\sum_{j\in J_{\beta}(\nu)}\left(a_{\iota^{-1}_{\nu}(w_{|\nu|+1})}^{(\nu,j)}\right)^2   +\mathrm{e}^{-\lambda_{\beta}^A(w)t}\sum_{j\in J_{\beta}(w)} a_{\iota^{-1}_w(x_1')}^{(w,j)}a_{\iota^{-1}_w(y_1')}^{(w,j)}\\
&=d^{|\beta|-1}N+\sum_{\nu \in \beta V_A, \nu<w} \mathrm{e}^{-\lambda_{\beta}^A(\nu)t}d^{|\nu|-1}(d-1)N-\mathrm{e}^{-\lambda_{\beta}^A(w)t}d^{|w|-1}N.
\end{align*}
\normalsize
After basic computations, we obtain the desired formula for $k_{\beta,t}$, and the remaining statements follow immediately. 
\end{proof}

Similarly with \cite{GGM25}, we obtain the following.

\begin{prop}
\label{heat_kernel_formula_groupoid}
Assume that all row-sums of $A$ are equal to $d\geq 2.$ Then, for $t>\t_c$, the heat operator $\mathrm{e}^{-t(\Delta+M_{L})}$ admits a kernel 
$$K_t\in L^1(G_A\times G_A,\mu_{G_A}\times \mu_{G_A})\cap C(G_A\times G_A).$$ 
Specifically, for $g_1\neq g_2\in G_A$, it is given by 
$$K_t(g_1,g_2)=\begin{cases}\mathrm{e}^{-t|\gamma|_c}k_{\gamma,t}(g_1,g_2), & \text{if  }\; g_1,g_2\in G_{\gamma}\\
0, & \text{otherwise}\end{cases}.$$ 
\end{prop}

\section{Heat trace asymptotics}\label{sec:sec5}

In this and the coming sections we focus on classical spectral geometry questions for the spectral triple $(\LC,L^2(G_A),D)$. We start by computing the heat traces $\Tr(\mathrm{e}^{-t\Delta_{C(\beta)}})$ of cylinder sets and then move on to the groupoid case $\Tr(\mathrm{e}^{-t|D|})$.

\subsection{Heat trace asymptotics on cylinder sets}

We start by computing heat traces on the cylinder sets, which will later be put together to results on $G_A$ via the bisections. We first deform the matrix $A$ via the potential $F_A$, to its deterministic part, namely the $N\times N$-matrix $D_A$ given by 
$$D_{A;i,j}:=
\begin{cases}
1, & P_{i,j}=1\\
0, & P_{i,j}< 1
\end{cases}.$$
A first but important observation is the following.

\begin{lemma}\label{lem:deterministic}
The matrix $D_A$ is nilpotent.
\end{lemma}

\begin{proof}
Since $A$ is primitive there is some $M\in \mathbb N$ such that $A^M$ has only positive entries. Therefore, we cannot have deterministic paths longer than $M$. 
\end{proof}

Now, observe that the potential $F_A$ is constant on cylinder sets of length two, hence we can view it as an $N\times N$-matrix $$F_{A;i,j}:=
\begin{cases}
F_A(C(ij)), & A_{i,j}=1\\
0, & A_{i,j}=0
\end{cases}.$$
Then, for $t\geq 0$ consider the $N\times N$-matrix $A(t)$ given by
\begin{equation}\label{eq:A(t)}
A(t)_{i,j}:= A_{i,j}\mathrm{e}^{-t\lambda_{A}F_{A;i,j}}.
\end{equation}
It is clear that each $A(t)$ is primitive since $A$ is primitive. Therefore, there is a right $r(A(t))$-eigenvector $u(t)$ and a left $r(A(t))$-eigenvector $v(t)$ of $A(t)$, normalised so that $\sum_{i=1}^{N} u(t)_i=1$ and $u(t)\cdot v(t)=1$.  Moreover, the following is true. 

\begin{lemma}
\label{lem:deform}
For $1\leq i,j\leq N$ it holds 
\begin{equation}\label{eq:deform1}
\lim_{t\to \infty} A(t)_{i,j}= D_{A;i,j}.
\end{equation}
Consequently, one has $\lim_{t\to \infty} r(A(t))=0.$
\end{lemma}

\begin{proof}
We focus on \eqref{eq:deform1}, as the second limit follows from the first one and Lemma \ref{lem:deterministic}. If $A_{i,j}=0$ then \eqref{eq:deform1} trivially holds. Assume now that $A_{i,j}=1$. If $P_{i,j}=1$, then $F_{A;i,j}=0$ and $A(t)_{i,j}=1=D_{A;i,j}$. Now, if $P_{i,j}<1$ we have that $F_{A;i,j}$ is bounded below by some $c>0$ that is independent of $i,j$. The result follows.
\end{proof}

\begin{prop}\label{prop:Laplacian_est}
We have that $\Tr(\mathrm{e}^{-t \Delta_{C(\beta)}})<\infty$ if and only if $t>\t_c$, where $\t_c\in (0,\infty)$ is the critical time introduced in Definition \ref{defn:crit_time_Ham}, namely  
$$\t_c=\lambda_{A}^{-1}\sup\limits_{\tau\in \mathcal{M}_{\sigma,e}}\frac{\ent_{\tau}}{\tau(F_A)}.$$ 
In particular, $\t_c $ is the unique solution of $r(A(t))=1$. Moreover, $$\lim_{t\to \t_c^+}(t-\t_c)\Tr(\mathrm{e}^{-t\Delta_{C(\beta)}})=-\frac{u(\t_c)_{\beta_{|\beta|}}}{r(A(\t_c))'}\sum_{j=1}^{N}\mathrm{e}^{-\t_c\lambda_A u_{j}}(\mathfrak{c}(j)-1) v(\t_c)_j.$$
\end{prop}

\begin{proof}
We aim to write $\Tr(\mathrm{e}^{-t\Delta_{C(\beta)}})$ in terms of $A(t)$, for $t\geq 0$. Also, let $H$ denote a bounded continuous function on $[0,\infty)$ that is real-analytic on $(0,\infty)$, which will be made clear from the context. Formally, we have 
\begin{align*}
\Tr(\mathrm{e}^{-t\Delta_{C(\beta)}})&=1+\sum_{\nu \in \overline{\beta V_{A}}} \mathrm{e}^{-t\lambda_{\beta}^{A}(\nu)}(\mathfrak{c}(\nu_{|\nu|})-1)\\
&=H(t)+\sum_{n\geq 0} \sum_{\nu \in \overline{\beta V_{A}^n}}\mathrm{e}^{-t\lambda_{\beta}^{A}(\nu)}(\mathfrak{c}(\nu_{|\nu|})-1)\\
&=H(t)+\sum_{n\geq 1} \sum_{\nu \in \overline{\beta V_{A}^n}}\mathrm{e}^{-t\lambda_{\beta}^{A}(\nu)}(\mathfrak{c}(\nu_{|\nu|})-1)\\
&=H(t)+\sum_{n\geq 1} \sum_{\nu \in \overline{\beta V_{A}^n}}\mathrm{e}^{-t\lambda_A u_{\nu_{|\nu|}}}(\mathfrak{c}(\nu_{|\nu|})-1)\prod_{k=0}^{n-1}\mathrm{e}^{-t\lambda_A F_{A;\nu_{|\beta|+k}, \nu_{|\beta|+k+1}}}\\
&=H(t)+\sum_{n\geq 1} \sum_{j:\mathfrak{c}(j)\geq 2}\mathrm{e}^{-t\lambda_A u_{j}}(\mathfrak{c}(j)-1)A(t)^n_{\beta_{|\beta|},j}\\
&=H(t)+\sum_{j:\mathfrak{c}(j)\geq 2}\mathrm{e}^{-t\lambda_A u_{j}}(\mathfrak{c}(j)-1) \left[\sum_{n\geq 1}A(t)^n_{\beta_{|\beta|},j}\right].
\end{align*}

Since $A(t)$ is primitive, the Perron--Frobenius theorem asserts that for $1\leq i,j\leq N$ it holds 
\begin{equation}\label{eq:deform2}
\Bigl \lvert  \frac{A(t)^n_{i,j}}{r(A(t))^n}-u(t)_iv(t)_j\Bigr \rvert = O\left(\left(\frac{|r_2(A(t))|}{r(A(t))}\right)^n\right),
\end{equation}
where $r_2(A(t))$ is the eigenvalue of $A(t)$ that has the largest modulus amongst all the other eigenvalues of $A(t)$ after $r(A(t))$. In particular, $|r_2(A(t))|<r(A(t)).$ This immediately means that $\Tr(\mathrm{e}^{-t \Delta_{C(\beta)}})<\infty$ if and only if $r(A(t))<1.$ 

We claim that there is some $\t_c'>0$ such that $r(A(t))<1$ if and only if $t>\t_c'$. We also aim to compute $\t_c'$ explicitly. To this end, for $t\geq 0$, let us denote for brevity the potential $-t\lambda_A F_A$ by $\phi_t$, which is a real-valued function on $\Omega_A$ and can also be viewed as an $N\times N$-matrix. Then, the topological pressure (see \cite{MU})
\begin{equation}\label{eq:deform3}
\pre(\phi_t)=\log r(A(t)).
\end{equation}
Indeed, consider the transfer operator $\mathcal{L}_{\phi_t}:C(\Omega_A,\mathbb R)\to C(\Sigma_A,\mathbb R)$ given by $$\mathcal{L}_{\phi_t}f(x)=\sum_{\sigma(y)= x}\mathrm{e}^{\phi_t(y)}f(y).$$ It is known that $\pre(\phi_t)=\log r(L_{\phi_t}).$ Moreover, for $n\in \mathbb N$ it holds
\begin{align*}
\mathcal{L}_{\phi_t}^n 1(x)&=\sum_{\sigma^n(y)=x}\mathrm{e}^{\sum_{k=0}^{n-1}\phi_t(\sigma^k(y))}\\
&=\sum_{j:jx_1\in V_A^{n+1}}\mathrm{e}^{\phi_t(j_1,j_2)}\ldots \mathrm{e}^{\phi_t(j_{n-1},j_n)}\mathrm{e}^{\phi_t(j_n,x_1)}\\
&=\sum_{j=1}^{N}A(t)^n(j,x_1).
\end{align*}
Since $C(\Omega_A,\mathbb R)$ is a lattice and $\mathcal{L}_{\phi_t}$ is positive we then get $$\|\mathcal{L}_{\phi_t}^n\|=\|\mathcal{L}_{\phi_t}^n1\|_{\infty}=\max\limits_{1\leq i\leq N}\sum_{j=1}^{N}A(t)^n(j,i)=\|A(t)^n\|_1.$$ Gelfand's formula then yields $r(\mathcal{L}_{\phi_t})=r(A(t))$, thus proving \eqref{eq:deform3}. Now the Variational Principle asserts that 
\begin{equation}\label{eq:deform4}
\pre(\phi_t)=\sup_{\tau\in \mathcal{M}_{\sigma}} \{\ent_{\tau}+\tau(\phi_t)\}.
\end{equation}
Here $\ent_{\tau}$ denotes the entropy of $\tau.$ Since the dynamics on $(\Omega_A,\sigma)$ are positively expansive, topologically mixing and $\phi_t$ is H{\"o}lder continuous, there is a unique $\tau_{\phi_t}\in \mathcal{M}_{\sigma,e}$ such that 
\begin{equation}\label{eq:deform5}
\pre(\phi_t)= \ent_{\tau_{\phi_t}}+\tau_{\phi_t}(\phi_t).
\end{equation}
Also, from \cite[Theorem 2.6.12]{MU} the function $[0,\infty)\to \mathbb R, \,\, t\mapsto \pre(\phi_t)$ is continuous and real-analytic on $(0,\infty)$. Further, from \cite[Proposition 2.6.13]{MU} we get that $$\pre(\phi_t)'=-\lambda_A\tau_{\phi_t}(F_A).$$ Then, Theorem \ref{lem:eigen_asym} implies $\pre(\phi_t)'<0$ for all $t>0$, hence $t\mapsto \pre(\phi_t)$ is strictly decreasing. Also, since $\pre(\phi_0)=\log \lambda_A >0$ and $\lim_{t\to \infty} \pre(\phi_t)=-\infty$ (see Lemma \ref{lem:deform} and \eqref{eq:deform3}), there is a unique $\t_c'>0$ such that $$\pre(\phi_{\t_c'})=0.$$ Then, $\pre(\phi_t)<0$ if and only if $t>\t_c'.$ From \eqref{eq:deform4} and \eqref{eq:deform5} we also get $$\t_c'=\t_c=\lambda_{A}^{-1}\sup\limits_{\tau\in \mathcal{M}_{\sigma,e}}\frac{\ent_{\tau}}{\tau(F_A)}.$$ 

We move on to show that  $\Tr(\mathrm{e}^{-t\Delta_{C(\beta)}})$ has a simple pole at $t=\t_c$. For $t\geq 0$, the functions $r(A(t)),u(t),v(t)$ are bounded real-analytic, and for $t>\t_c$ using \eqref{eq:deform2} we have 
\begin{align*}
\Tr(\mathrm{e}^{-t\Delta_{C(\beta)}})&=H(t)+\sum_{j:\mathfrak{c}(j)\geq 2}\mathrm{e}^{-t\lambda_A u_{j}}(\mathfrak{c}(j)-1) \left[\sum_{n\geq 1}A(t)^n_{\beta_{|\beta|},j}\right]\\
&= H(t)+\sum_{j:\mathfrak{c}(j)\geq 2}\mathrm{e}^{-t\lambda_A u_{j}}(\mathfrak{c}(j)-1) u(t)_{\beta_{|\beta|}}v(t)_j\left[\sum_{n\geq 0}r(A(t))^n \right]\\
&= H(t)+\sum_{j:\mathfrak{c}(j)\geq 2}\mathrm{e}^{-t\lambda_A u_{j}}(\mathfrak{c}(j)-1) u(t)_{\beta_{|\beta|}}v(t)_j \left[1-r(A(t))\right]^{-1}.
\end{align*}
Since $1-r(A(t))$ is analytic around $t=\t_c$, $1-r(A(\t_c))=0$ and $r(A(\t_c))'\neq 0$ as $r(A(\t_c))'=\pre(\phi_{\t_c})'<0$, we get that $$1-r(A(t))=-r(A(\t_c))'(t-\t_c)+O((t-\t_c)^2).$$ Consequently, $$\lim_{t\to \t_c^+}(t-\t_c)\Tr(\mathrm{e}^{-t\Delta_{C(\beta)}})=-\frac{1}{r(A(\t_c))'}\sum_{j:\mathfrak{c}(j)\geq 2}\mathrm{e}^{-\t_c\lambda_A u_{j}}(\mathfrak{c}(j)-1) u(\t_c)_{\beta_{|\beta|}}v(\t_c)_j.$$
\end{proof}

As a corollary to Proposition \ref{prop:Laplacian_est}, or rather its proof, we conclude the following from a short computation resulting in local heat trace asymptotics. First, recall that $A(t)_{i,j}= A_{i,j}\mathrm{e}^{-t\lambda_{A}F_{A;i,j}}$, which we consider as an entire $N\times N$-matrix valued function of $t\in \C$. Consider the set
\begin{equation}\label{eq:singularities}
\mathfrak{T}_A:=\{t\in \C: \; 1\in \mathrm{Spec}(A(t))\}
\end{equation}
Since $A(t)$ is entire and $\mathfrak{T}_A\neq \C$ by Lemma \ref{lem:deform}, we have that $(1-A(t))^{-1}$ is a meromorphic function of $t\in \C$, so $\mathfrak{T}_A\subseteq \C$ is a discrete subset of $\mathbb C$.

\begin{cor}
\label{cortoLaplacian_est}
Let $\beta \in V_A\setminus \{\o\}$ and $w\in \beta V_A$. Then, the function $t\mapsto  \Tr(\chi_{C(w)}\mathrm{e}^{-t\Delta_{C(\beta)}})$ defined for $t>\t_c$ extends to a meromorphic function in $t\in \C$, with poles of order at most $N$ and situated in the discrete set $\mathfrak{T}_A$. Moreover, as $t\to t_c^+$,
\begin{align*}
\Tr&(\chi_{C(w)}\mathrm{e}^{-t\Delta_{C(\beta)}})=\\
=&
\begin{cases}
\Tr(\mathrm{e}^{-t\Delta_{C(\beta)}}), &\mbox{if $C(w)= C(\beta)$},\\
\Tr(\mathrm{e}^{-t\Delta_{C(\beta)}})\frac{u(t)_{w_{|w|}}}{u(t)_{\beta_{|\beta|}}}\prod_{k=0}^{|w|-|\beta|-1}\mathrm{e}^{-t \lambda_A F_{A;w_{|\beta|+k},w_{|\beta|+k+1}}}+H_w(t) &\mbox{if $C(w)\subsetneq C(\beta)$},
\end{cases},
\end{align*}
where $H_w(t)$ is bounded continuous on $[0,\infty)$ and real-analytic on $(0,\infty)$.
\end{cor}

\begin{proof}
From the proof of Proposition \ref{prop:Laplacian_est}, we see that for an entire function $H$,
\begin{equation}\label{eq:heat_ord_poles}
\Tr(\mathrm{e}^{-t\Delta_{C(\beta)}})=H(t)+\sum_{j:\mathfrak{c}(j)\geq 2}\mathrm{e}^{-t\lambda_A u_{j}}(\mathfrak{c}(j)-1) \left[A(t)(1-A(t))^{-1}\right]_{\beta_{|\beta|},j}.
\end{equation}
Therefore,  $t\mapsto  \Tr(\mathrm{e}^{-t\Delta_{C(\beta)}})$ extends to a meromorphic function in $t\in \C$ whose poles are of order at most $N$ and situated in the discrete set $\mathfrak{T}_A$.

The only non-trivial case to consider is when $C(w)\subsetneq C(\beta)$. Also, without loss of generality we can assume that $w\in \overline{\beta V_A}$. As in Proposition \ref{prop:Laplacian_est} and Theorem \ref{thm:Heat_kernel_formula}, for $t>\t_c$ we compute that
\footnotesize
\begin{align*}
\Tr(\chi_{C(w)}\mathrm{e}^{-t\Delta_{C(\beta)}})
&=\int_{C(w)}k_{\beta,t}(x,x)\d \mu_c(x)\\
&=\int_{C(w)}h_{\beta}(x)^2 \d \mu_c(x)
+ \sum_{\nu\in \overline{\beta V_A}} \mathrm{e}^{-\lambda_{\beta}^A(\nu)t}
\sum_{j\in J_{\beta}(\nu)}
\int_{C(w)}h_{\nu,j}(x)^2\d \mu_c(x)\\
&=\mu_c(C(w))\mu_c(C(\beta))^{-1}
+\sum_{\nu < w}\mathrm{e}^{-\lambda_{\beta}^A(\nu)t}
\sum_{j\in J_{\beta}(\nu)}
\left(a^{(\nu,j)}_{\iota_{\nu}^{-1}(w_{|\nu|+1})}\right)^2
\mu_c(C_{w_1\ldots w_{|\nu|+1}})\\
&\qquad +\sum_{\nu \in \overline{w V_{A}}}
\mathrm{e}^{-t\lambda_{\beta}^{A}(\nu)}
(\mathfrak{c}(\nu_{|\nu|})-1)\\
&=H_w(t)+\sum_{n\geq 0} \sum_{\nu \in \overline{w V_{A}^n}}
\mathrm{e}^{-\lambda_{\beta}^{A}(\nu)t}
(\mathfrak{c}(\nu_{|\nu|})-1)\\
&=H_w(t)+\sum_{n\geq 1} \sum_{\nu \in \overline{w V_{A}^n}}
\mathrm{e}^{-\lambda_{\beta}^{A}(\nu)t}
(\mathfrak{c}(\nu_{|\nu|})-1)\\
&= H_w(t)+\left(\prod_{k=0}^{|w|-|\beta|-1}\mathrm{e}^{-t \lambda_A F_{A;w_{|\beta|+k},w_{|\beta|+k+1}}}\right)\sum_{n\geq 1}\sum_{\nu \in \overline{wV_A^n}}\mathrm{e}^{-\lambda_w^A(\nu)t}(\mathfrak{c}(\nu_{|\nu|})-1)\\
&=H_w(t)+\Tr(\mathrm{e}^{-t\Delta_{C(\beta)}})\frac{u(t)_{w_{|w|}}}{u(t)_{\beta_{|\beta|}}}\prod_{k=0}^{|w|-|\beta|-1}\mathrm{e}^{-t \lambda_A F_{A;w_{|\beta|+k},w_{|\beta|+k+1}}}.
\end{align*}
\normalsize
Similarly with \eqref{eq:heat_ord_poles}, for an entire function $H_w$ we can also write,
\footnotesize
\begin{align*}
\Tr(\chi_{C(w)}&\mathrm{e}^{-t\Delta_{C(\beta)}})=\\
=&H_w(t)+\sum_{j:\mathfrak{c}(j)\geq 2}\mathrm{e}^{-t\lambda_A u_{j}}(\mathfrak{c}(j)-1) \left[A(t)(1-A(t))^{-1}\right]_{w_{|w|},j}\prod_{k=0}^{|w|-|\beta|-1}\mathrm{e}^{-t \lambda_A F_{A;w_{|\beta|+k},w_{|\beta|+k+1}}},
\end{align*}
\normalsize
so we can conclude that $t\mapsto  \Tr(\chi_{C(w)}\mathrm{e}^{-t\Delta_{C(\beta)}})$ extends to a meromorphic function in $t\in \C$ whose poles are of order at most $N$ and situated in the discrete set $\mathfrak{T}_A$. The proof is now complete.
\end{proof}

As the potential $\phi_{\t_c}:=-\t_c\lambda_AF_A$ is constant on cylinder sets of length two, the associated equilibrium measure $\tau_{\phi_{\t_c}}$ is a Markov measure. Namely, it is given by the distribution $p(\t_c):=\left(p(\t_c)_j=u(\t_c)_jv(\t_c)_j\right)_{j=1}^{N}$ and the transition probabilities $$P(\t_c)_{i,j}:=A(\t_c)_{i,j}\frac{u(\t_c)_j}{u(\t_c)_i}.$$ Consequently, the residue at $t=\t_c$ is a weighted average of the function $\mathfrak{c}$. 

\begin{cor}
It holds that $$\lim_{t\to \t_c^+}(t-\t_c)\Tr(\mathrm{e}^{-t\Delta_{C(\beta)}})=-\frac{u(\t_c)_{\beta_{|\beta|}}}{r(A(\t_c))'}\mathbb E_{\tau_{\phi_{\t_c}}}\left((\mathfrak{c}-1) G_{\t_c}\right),$$ where the weight $G_{\t_c}:\{1,\ldots,N\}\to (0,\infty)$ is given by $G_{\t_c}(j):=\mathrm{e}^{-\t_c\lambda_Au_j}u(\t_c)_j^{-1}.$
\end{cor}

\begin{example}\label{ex:out-reg}
Assume $A$ has row-sums equal to some $d\geq 2$. Then, from Proposition \ref{prop:sA2} we see that $F_A$ is constant. In fact, viewing it as an $N\times N$-matrix we have 
$$F_{A;i,j}=
\begin{cases}
N^{-1}(1-d^{-1}), & A_{i,j}=1\\
0, & A_{i,j}=0
\end{cases}.$$
As a result, for $t\geq 0$ the matrix $A(t)$ has row-sums equal to $d\mathrm{e}^{-tN^{-1}(d-1)}$ and hence $$r(A(t))=d\mathrm{e}^{-tN^{-1}(d-1)}.$$ Also, the right $r(A(t))$-eigenvector $u(t)$ of $A(t)$ can be normalised so that $u(t)_j=N^{-1}$. This means $\sum_{j=1}^{N} v(t)_j=N$. We also have that $\t_c=N(d-1)^{-1}\log d$ (see Remark \ref{rem:crit}) and $$\lim_{t\to \t_c^+}(t-\t_c)\Tr(\mathrm{e}^{-t\Delta_{C(\beta)}})=Nd^{-\frac{d}{d-1}}.$$ Note that $\t_c$ is exactly the threshold $\df \delta^{-1}$ appearing in Theorem \ref{thm:Heat_kernel_formula}. 
\end{example}

\subsection{Heat trace asymptotics on the groupoid}
\label{heattracegroupdopoad}

We are now ready to compute the asymptotics of the heat traces of Hamiltonians on the whole groupoid $G_A$. For $i,j\in \{1,\ldots,N\}$ we let $V_A(i,j)$ denote the set of admissible words of arbitrary length which start with $i$ and end with $j$. Also, by $V_A^n(i,j)\subset V_A(i,j)$ denote the words of length $n\in \mathbb N$.

\begin{lemma}
\label{lem:Potential_est}
For $i,j\in \{1,\ldots,N\}$ and $t>0$ define 
\begin{equation}
\label{lknlanda}
F(t,i,j):=\sum_{\alpha\in V_A(i,j)} \mathrm{e}^{-t|\alpha|}.
\end{equation}
The series \eqref{lknlanda} converges if and only if $t>\log \lambda_A$ and 
$$\lim_{t\to \log \lambda_A^+}(t-\log \lambda_A)F(t,i,j)=\frac{u_i v_j}{\lambda_A}.$$
Moreover, $t\mapsto F(t,i,j)$ extends to a meromorphic function of $t\in \C$ with poles being of order at most $N$ situated in the set 
$$\mathrm{LSpec}(A):=\{t\in \C: \mathrm{e}^t\in \mathrm{Spec}(A)\}.$$
\end{lemma}

\begin{proof}
We compute that
\begin{align*}
F(t,i,j)&= \sum_{n=1}^{\infty}\mathrm{e}^{-tn}\left(\sum_{\alpha\in V_A^n(i,j)}1 \right)= \delta_{i,j} \mathrm{e}^{-t}+\sum_{n=2}^{\infty} \mathrm{e}^{-tn} A_{i,j}^{n-1}=\\
=&\mathrm{e}^{-t}(\delta_{i,j}-1+\sum_{n=0}^{\infty} \mathrm{e}^{-tn} A_{i,j}^{n})=\mathrm{e}^{-t}([(1-\mathrm{e}^{-t}A)^{-1}]_{i,j}).
\end{align*}
We see that $F(t,i,j)$ extends to a meromorphic function of $t\in \C$ with poles being of order at most $N$ situated in $\mathrm{LSpec}(A)$.

Again, consider a function $H(t)$ analytic for $t>\log \lambda_A -\varepsilon,$ where $\varepsilon>0$ is small enough. Using the asymptotics \eqref{eq:deform2} of the Perron--Frobenius theorem we get that

\begin{align*}
F(t,i,j)&=  \delta_{i,j} \mathrm{e}^{-t}+\sum_{n=2}^{\infty} \mathrm{e}^{-tn} A_{i,j}^{n-1}\\
&=\delta_{i,j}\mathrm{e}^{-t}+\sum_{n=2}^{\infty}\mathrm{e}^{-tn}\lambda_A^{n-1}\left(\lambda_A^{-n+1}A_{i,j}^{n-1}-u_iv_j+u_iv_j\right)\\
&= H(t)+\sum_{n=2}^{\infty}(\mathrm{e}^{-t+\log \lambda_A})^n \lambda_A^{-1}u_iv_j\\
&=H(t)+\lambda_A^{-1}u_iv_j\mathrm{e}^{-2(t-\log \lambda_A)}(1-\mathrm{e}^{-t+\log \lambda_A})^{-1}.
\end{align*}
The result follows. 
\end{proof}

\begin{remark}\label{rem:synch}
At this point it becomes clear why we had to normalise $\widetilde{D}$ from Subsection \ref{subsec:ham} and obtain $D$ from Definition \ref{defn:crit_time_Ham}. Following Proposition \ref{prop:Laplacian_est}, the singularity introduced by $\Delta$ happens at $t=\t_c$, whereas the one from $M_{\widetilde{L}}$ occurs at $t=\log \lambda_A$ as we see from Lemma \ref{lem:Potential_est}. However, $\t_c >\log \lambda_A$ and therefore it seems appropriate to synchronise the singularities by rescaling $\widetilde{L}$. In other words, by considering the heat trace $\Tr(\mathrm{e}^{-t|\widetilde{D}|})$ for $t>\t_c$ would miss completely the dynamical information encoded in $\widetilde{L}.$
\end{remark}
 Recall that $$|\cdot|_c:=\frac{\log \lambda_A}{\t_c} |\cdot|,\qquad L:=\frac{\log \lambda_A}{\t_c}\widetilde{L},\qquad D:=-\Delta + (2P_A-1)M_{L}.$$ 

For $j\in \{1,\ldots,N\}$, denote $\lim_{t\to \t_c^+}(t-\t_c)\Tr(\mathrm{e}^{-t\Delta_{C(j)}})$ by $\Res\limits_{t=\t_c}\Tr(\mathrm{e}^{-t\Delta_{C(j)}})$, which is computed explicitly in Proposition \ref{prop:Laplacian_est}. Also, let 
$$V(j):=\frac{{\t_c^2}}{\lambda_A^3\log^2\lambda_A}\sum_{i=1}^{N}\sum_{k\neq i}A_{i,j}A_{k,j}v_k v_i.$$ Note that $V(j)=0$ exactly when $j$ has in-degree one. Then, the following holds.

\begin{thm}
\label{thm:HeatTr}
We have that $\Tr(\mathrm{e}^{-t|D|})<\infty$ if and only if $t>\t_c$, and in fact $t\mapsto \Tr(\mathrm{e}^{-t|D|})$ extends to a meromorphic function of $t\in \C$ with poles of order at most $3N$ located in the set $\mathfrak{T}_A\cup (\t_c\mathrm{LSpec}(A)/\log \lambda_A)$, where $\mathrm{LSpec}(A):=\{t\in \C: \mathrm{e}^t\in \mathrm{Spec}(A)\}$. Moreover, 
$$\lim_{t\to \t_c^+} (t-\t_c)^3\Tr(\mathrm{e}^{-t|D|})=\sum_{j=1}^{N}\Res\limits_{t=\t_c}\Tr(\mathrm{e}^{-t\Delta_{C(j)}})V(j).$$ 
\end{thm}

\begin{proof}
First, note that for every $\beta \in V_A\setminus \{{\o}\}$, one has $$\Tr(\mathrm{e}^{-t\Delta_{C(\beta)}})=\Tr(\mathrm{e}^{-t\Delta_{C(\beta_{|\beta|})}}).$$ Consequently,

\begin{align*}
\Tr(\mathrm{e}^{-t|D|})&=\sum_{\alpha.\beta \in I_A}\Tr(\mathrm{e}^{-t(\Delta_{\alpha.\beta}+|\alpha|_c+|\beta|_c)})\\
&=\sum_{\alpha.\beta \in I_A} \mathrm{e}^{-t(|\alpha|_c+|\beta|_c)}\Tr(\mathrm{e}^{-t\Delta_{\alpha.\beta}})\\
&= \sum_{\alpha.\beta \in I_A}\mathrm{e}^{-t(|\alpha|_c+|\beta|_c)}\Tr(\mathrm{e}^{-t\Delta_{C(\beta_{|\beta|})}})\\
&=\sum_{j=1}^{N}\Tr(\mathrm{e}^{-t\Delta_{C(j)}})\sum_{\alpha.\beta\in I_A, \beta_{|\beta|}=j}\mathrm{e}^{-t(|\alpha|_c+|\beta|_c)}.
\end{align*}
From Proposition \ref{prop:Laplacian_est} and Lemma \ref{lem:Potential_est} we see that $\Tr(\mathrm{e}^{-t|D|})<\infty$ if and only if $t>\t_c$. Our goal now is to find asymptotics at $t=\t_c$. 

To this end, recall the function $F$ from Lemma \ref{lem:Potential_est}, and for $j\in \{1,\ldots,N\}$ define 
\begin{equation}\label{eq:HeatTr0}
F_c(t,j):=\sum_{i=1}^{N}F(t(\t_c^{-1} \log \lambda_A), i,j).
\end{equation}
It is clear that $F_c(t,j)<\infty$ if and only if $t>\t_c.$ Also, since $\sum_{i=1}^{N} u_i=1$ we have that 
\begin{equation}\label{eq:HeatTr1}
\lim_{t\to \t_c^+}(t-\t_c)F_c(t,j)=\frac{v_j \t_c}{\lambda_A \log \lambda_A}.
\end{equation}
Further, for $\beta\in V_A$ with $|\beta|\geq 2$ we have
\begin{align}\label{eq:HeatTr2}
\sum_{\alpha: \alpha.\beta\in I_A} \mathrm{e}^{-t|\alpha|_c}&=1+\sum_{{\o}\neq \alpha: \alpha.\beta\in I_A} \mathrm{e}^{-t|\alpha|_c}\\ \nonumber
&=1+\sum_{\alpha\in V_A\setminus \{{\o}\}}(1-\delta_{\alpha_{|\alpha|},\beta_{|\beta|-1}} )A_{\alpha_{|\alpha|},\beta_{|\beta|}} \mathrm{e}^{-t|\alpha|_c}\\ \nonumber
&=1+\sum_{k\neq \beta_{|\beta|-1}}A_{k,\beta_{|\beta|}}F_c(t,k).
\end{align}
In particular, the sum $\sum_{\alpha: \alpha.\beta\in I_A} \mathrm{e}^{-t|\alpha|_c}$ depends only on the last two letters of $\beta$. Also, let $W_A(i,j)$ denote all admissible words $\beta$ such that $|\beta|\geq 2$ and $\beta_{|\beta|-1}=i, \beta_{|\beta|}=j$. Then, 
\begin{equation}\label{eq:HeatTr3}
\sum_{\beta\in W_A(i,j)}\mathrm{e}^{-t|\beta|_c}=A_{i,j}\mathrm{e}^{-t(\t_c^{-1}\log \lambda_A)} F_c(t,i).
\end{equation}
Combining \eqref{eq:HeatTr2} and \eqref{eq:HeatTr3} we obtain that 
\small
\begin{align*}
\noindent \sum_{\alpha.\beta\in I_A, \beta_{|\beta|}=j}\mathrm{e}^{-t(|\alpha|_c+|\beta|_c)}&=\\
\mathrm{e}^{-t(\t_c^{-1}\log \lambda_A)}&\left(1+\sum_{k=1}^{N}A_{k,j}F_c(t,k)+\sum_{i=1}^{N}A_{i,j}\left(1+\sum_{k\neq i}A_{k,j}F_c(t,k)\right)F_c(t,i)\right).
\end{align*}
\normalsize
The limit statement as $t\to t_c^+$ now follows from \eqref{eq:HeatTr1} and the meromorphic extension statement follows from Corollary \ref{cortoLaplacian_est} and Lemma \ref{lem:Potential_est}.
\end{proof}

\section{Local heat trace asymptotics, or all roads lead to the KMS-state}

We now turn our attention to computing local heat trace asymptotics. We shall study three different types of localisations of the heat trace, and their common feature is that the leading term behaviour determines the KMS-state $\varphi_A$. In Subsection \ref{localI}, we study the local heat traces $\Tr(a\mathrm{e}^{-t|D|})$. This object is called a local heat trace following the convention on manifolds where the asymptotics of objects like $\Tr(a\mathrm{e}^{-t|D|})$ is governed by integrals of $a$ against curvature-like terms. In Subsection \ref{localII} we study Dixmier traces, regularized versions of traces, that by the analogy with manifolds can also be called local. Finally, in Subsection \ref{localIII} we study local heat traces on the positive part $\Tr(P_Da\mathrm{e}^{-t|D|})$, where $P_D=\chi_{[0,\infty)}(D)$. This object was studied in detail in \cite{GRU} and we revisit it here for completeness, to drive the point home that all roads lead to the KMS-state $\varphi_A$.

\subsection{Local heat trace asymptotics for $|D|$}
\label{localI}

Our next goal is to compute the leading asymptotic behavior of $\Tr(a\mathrm{e}^{-t|D|})$, for $a\in \LC$. In particular, the states 
\begin{equation}
\label{alknaldna}
\varphi_t(a):=\frac{\Tr(a\mathrm{e}^{-t|D|})}{\Tr(\mathrm{e}^{-t|D|})}, \quad t>\t_c,
\end{equation}
 will be shown to have a well defined limit as $t\to \t_c^+$. Following the proof of Theorem \ref{thm:HeatTr} we arrive at the following.

\begin{thm}
\label{thm:HeatTrwitha}
For $a\in \LC$, we have that $\Tr(a\mathrm{e}^{-t|D|})$ converges for $t>\t_c$. Moreover, for $a=S_w S_\nu^*$ we have that 
\begin{enumerate}
\item The function $t\mapsto \Tr(a\mathrm{e}^{-t|D|})$ extends to a meromorphic function of $t\in \C$ with poles of order at most $3N$ located in the set $\mathfrak{T}_A\cup (\t_c\mathrm{LSpec}(A)/\log \lambda_A)$, where $\mathrm{LSpec}(A):=\{t\in \C: \mathrm{e}^t\in \mathrm{Spec}(A)\}$. 
\item If at least one of $w$ or $\nu$ is non-empty, then as $t\to \t_c^+$,
\begin{equation}
\label{lknalkdnaljbn}
\Tr(S_w S_\nu^*\mathrm{e}^{-t|D|})=\delta_{w,\nu}\lambda_Au_{w_{|w|}}\mathrm{e}^{-t|w|_c}\Tr(\mathrm{e}^{-t|D|})+O((t-\t_c)^{-2}).
\end{equation}
\end{enumerate}
In particular, the family of states \eqref{alknaldna} converge as $t\to \t_c^+$ to the standard KMS-state $\varphi_A$ on $O_A$, given by 
$$\varphi_A(S_w S_\nu^*):= \delta_{w,\nu}\lambda_A^{-|w|+1}u_{w_{|w|}},$$
when at least one of $w$ or $\nu$ is non-empty.
\end{thm}

\begin{proof}
First, note that for every distinct $w,\nu \in V_A$, one has 
$$\Tr(S_w S_\nu^*\mathrm{e}^{-t|D|})=0.$$
So we will assume that $w=\nu$ and $a=S_w S_w^*$. Theorem \ref{thm:HeatTr} already proves the case that $w$ is empty, so we assume that $w$ is non-empty.

We have that 
$$\Tr_{L^2(G_A)}(S_w S_w^*\mathrm{e}^{-t|D|})=\sum_{\alpha.\beta\in I_A} \mathrm{e}^{-t(|\alpha|_c+|\beta|_c)} \Tr_{L^2(C(\beta))}\left(\kappa_{\alpha.\beta}^*(\chi_{C(w)}) \mathrm{e}^{-t\Delta_{C(\beta)}}\right)$$
where $\kappa_{\alpha.\beta}:C(\beta)\to \Omega_A$ is defined through the restricted range and source maps $r_{\alpha.\beta}$ and $s_{\alpha.\beta}$ on $G_{\alpha.\beta}$ as
$$\kappa_{\alpha.\beta}:=r_{\alpha.\beta}\circ s_{\alpha.\beta}^{-1}.$$
We compute that 
$$\kappa_{\alpha.\beta}^*(\chi_{C(w)})=\chi_{\kappa_{\alpha.\beta}^{-1}(C(w))},$$
and for $\alpha.\beta\in I_A$
\begin{align*}
\kappa_{\alpha.\beta}^{-1}(C(w))=&\{y\in C(\beta): \, \alpha \sigma^{|\beta|-1}(y)\in C(w)\}=\\
=&
\begin{cases}
C_\beta, \; &\mbox{if $\alpha\in w V_A$},\\
C_{\beta\bar{\bar{w}}}\; &\mbox{if $w\in \alpha V_A\setminus \{\alpha\}$ with $w=\alpha\bar{w}$ and $\bar{w}=\beta_{|\beta|}\bar{\bar{w}}$},\\
\emptyset, \; &\mbox{otherwise}.\end{cases}
\end{align*}

Consequently,
\begin{align*}
\Tr_{L^2(G_A)}(S_w S_w^*\mathrm{e}^{-t|D|})=&\sum_{\substack{\alpha.\beta\in I_A\\ \alpha\in w V_A}} \mathrm{e}^{-t(|\alpha|_c+|\beta|_c)} \Tr_{L^2(C(\beta))}\left(\mathrm{e}^{-t\Delta_{C(\beta)}}\right)\\
&+\sum_{\substack{\beta\in V_A\setminus \{\o\}\\ \beta_{|\beta|}=w_1}} \mathrm{e}^{-t|\beta|_c} \Tr_{L^2(C(\beta))}\left(\chi_{C(\beta\bar{\bar{w}})} \mathrm{e}^{-t\Delta_{C(\beta)}}\right)\\
&+\sum_{l=1}^{|w|-1}\sum_{\substack{\beta\in V_A\setminus \{\o\}\\ w_{l+1}=\beta_{|\beta|}, \\ w_l\neq \beta_{|\beta|-1}}} \mathrm{e}^{-t\left(l \t_c^{-1}\log \lambda_A +|\beta|_c\right)} \Tr_{L^2(C(\beta))}\left(\chi_{C(\beta\bar{\bar{w}})} \mathrm{e}^{-t\Delta_{C(\beta)}}\right).
\end{align*}
where the last two terms arise from $w\in \alpha V_A\setminus \{\alpha\}$, where we reconstruct $\alpha=w_1\cdots w_l$ from $l=|\alpha|$. At this stage, in light of Corollary \ref{cortoLaplacian_est}, we note that a similar proof as that of Theorem \ref{thm:HeatTr} combined with the computations we do next, to compute the limit as $t\to \t_c^+$, in fact proves that $t\mapsto \Tr(a\mathrm{e}^{-t|D|})$ extends to a meromorphic function of $t\in \C$ with poles of order at most $3N$ located in the set $\mathfrak{T}_A\cup (\t_c\mathrm{LSpec}(A)/\log \lambda_A)$. 

For the limit as $t\to \t_c^+$, it follows from Corollary \ref{cortoLaplacian_est} and Lemma \ref{lem:Potential_est} that
\begin{align*}
\sum_{\substack{\beta\in V_A\setminus \{\o\}\\ \beta_{|\beta|}=w_1}} \mathrm{e}^{-t|\beta|_c} \Tr_{L^2(C(\beta))}\left(\chi_{C(\beta\bar{\bar{w}})} \mathrm{e}^{-t\Delta_{C(\beta)}}\right)=&O((t-\t_c)^{-2}), \quad\mbox{and}\\
\sum_{l=1}^{|w|-1}\sum_{\substack{\beta\in V_A\setminus \{\o\}\\ w_{l+1}=\beta_{|\beta|}, \\ w_l\neq \beta_{|\beta|-1}}} \mathrm{e}^{-t\left(l \t_c^{-1}\log \lambda_A +|\beta|_c\right)} \Tr_{L^2(C(\beta))}\left(\chi_{C(\beta\bar{\bar{w}})} \mathrm{e}^{-t\Delta_{C(\beta)}}\right)=&O((t-\t_c)^{-2}).
\end{align*}
As a result,
\begin{align}
\nonumber
\Tr_{L^2(G_A)}(S_w S_w^*&\mathrm{e}^{-t|D|})=\sum_{\substack{\alpha.\beta\in I_A\\ \alpha\in w V_A}} \mathrm{e}^{-t(|\alpha|_c+|\beta|_c)} \Tr_{L^2(C(\beta))}\left(\mathrm{e}^{-t\Delta_{C(\beta)}}\right)+O((t-\t_c)^{-2})\\
\label{computedinisdn}
=&\sum_{j=1}^N\Tr_{L^2(C(j))}\left(\mathrm{e}^{-t\Delta_{C(j)}}\right)\sum_{\substack{\alpha.\beta\in I_A\\ \alpha\in w V_A,\\ \beta_{|\beta|}=j}} \mathrm{e}^{-t(|\alpha|_c+|\beta|_c)} +O((t-\t_c)^{-2}).
\end{align}

We now utilize the function $F_c(t, i,j):=F(t(\t_c^{-1} \log \lambda_A), i,j)$ and the function $F_c(t,j):=\sum_{i=1}^{N}F(t(\t_c^{-1} \log \lambda_A), i,j)$ used in the proof of Theorem \ref{thm:HeatTr}. For a fixed $\beta\in V_A$ with $|\beta|\geq 2$ we have
\begin{align}\label{eq:HeatTr2new}
\sum_{\substack{\alpha.\beta\in I_A\\ \alpha\in w V_A}} \mathrm{e}^{-t|\alpha|_c}&=\mathrm{e}^{-t|w|_c}+\sum_{\substack{\alpha.\beta\in I_A\\ \alpha\in w (V_A\setminus \{\o\})}} \mathrm{e}^{-t|\alpha|_c}\\ \nonumber
&=\mathrm{e}^{-t|w|_c}+\sum_{\alpha\in  w (V_A\setminus \{\o\})}(1-\delta_{\alpha_{|\alpha|},\beta_{|\beta|-1}} )A_{\alpha_{|\alpha|},\beta_{|\beta|}} \mathrm{e}^{-t|\alpha|_c}\\ \nonumber
&=\mathrm{e}^{-t|w|_c}\left(1+\sum_{\alpha\in  V_A\setminus \{\o\}}(1-\delta_{\alpha_{|\alpha|},\beta_{|\beta|-1}} )A_{w_{|w|},\alpha_1}A_{\alpha_{|\alpha|},\beta_{|\beta|}} \mathrm{e}^{-t|\alpha|_c}\right)\\ \nonumber
&=\mathrm{e}^{-t|w|_c}\left(1+\sum_{k\neq \beta_{|\beta|-1}}\sum_{l=1}^NA_{w_{|w|},l}A_{k,\beta_{|\beta|}}F_c(t,l,k)\right).
\end{align}
In particular, the sum $\sum_{\substack{\alpha.\beta\in I_A\\ \alpha\in w V_A}} \mathrm{e}^{-t|\alpha|_c}$ depends only on the last two letters of $\beta$ for fixed $w$. Also, let $W_A(i,j)$ denote all admissible words $\beta$ such that $|\beta|\geq 2$ and $\beta_{|\beta|-1}=i, \beta_{|\beta|}=j$. Then, combining \eqref{eq:HeatTr2new} with \eqref{eq:HeatTr1} and \eqref{eq:HeatTr3} we obtain that as $t\to \t_c^+$
\small
\begin{align*}
\sum_{\substack{\alpha.\beta\in I_A\\ \alpha\in w V_A,\\ \beta_{|\beta|}=j}} &\mathrm{e}^{-t(|\alpha|_c+|\beta|_c)} =\\
=&\mathrm{e}^{-t|w|_c}\mathrm{e}^{-t(\t_c^{-1}\log \lambda_A)}&\sum_{i=1}^{N}\sum_{k\neq i}\sum_{l=1}^NA_{i,j}A_{w_{|w|},l}A_{k,j}F_c(t,l,k)F_c(t,i)+O((t-\t_c)^{-1}).
\end{align*}
\normalsize
Consequently, we have that $$\lim_{t\to \t_c^+}(t-\t_c)^2\sum_{\substack{\alpha.\beta\in I_A\\ \alpha\in w V_A,\\ \beta_{|\beta|}=j}} \mathrm{e}^{-t(|\alpha|_c+|\beta|_c)}= \lambda_A^{-|w|+1}u_{w_{|w|}}V(j),$$

which together with \eqref{computedinisdn} and Corollary \ref{cortoLaplacian_est} implies that \eqref{lknalkdnaljbn} holds. 
\end{proof}

\subsection{Noncommutative integral formula for the canonical KMS-state}
\label{localII}

We move on to describe the KMS-state of the gauge action on $O_A$ in terms of Dixmier traces. For a compact operator $T$, we let $(\mu_k(T))_k$ denote its sequence of singular values ordered decreasingly. We define the Dixmier--Macaev ideal of compact operators $\mathcal{M}_{1,\infty}$ on a separable Hilbert space $\mathcal{H}$ as 
$$\mathcal{M}_{1,\infty}(\mathcal{H}):=\left\{T\in \mathbb{K}(\mathcal{H}): \sup_K \frac{1}{\log(2+K)}\sum_{n=1}^K \mu_n(T)<\infty\right\}.$$
A state $\omega\in \ell^{\infty}(\N)^*$ is said to be an extended limit if $\omega|_{c_0(\N)}=0$, or in other words $\omega$ is an extension of the limiting functional from the subspace of convergent sequences to all of $\ell^{\infty}(\N)$. If $\omega((x_n)_n)=\omega((x_{2n})_n)$ we say that $\omega$ is dilation invariant. It is well-known, see for instance \cite{LSZ}, that a dilation invariant extended limit $\omega$ defines a trace 
$$\Tr_\omega:\mathcal{M}_{1,\infty}(\mathcal{H})\to \C,$$
by setting $\Tr_\omega(T):=\omega((\frac{1}{\log(2+K)}\sum_{n=1}^K \mu_n(T))_K)$ for $T\geq 0$ and extended by linearity. The trace $\Tr_\omega$ is called a Dixmier trace and was first constructed by Dixmier, and later elevated to the role of an integral in Connes' noncommutative geometry \cite{connesredbook}. For more context, see \cite{connesredbook,LSZ}.

\begin{thm}
\label{thm:Dixmier}
Let $\omega$ be a dilation invariant extended limit on $\ell^{\infty}(\N)$. Also, recall the constant
$C=\lim_{t\to \t_c} (t-\t_c)^3\Tr(\mathrm{e}^{-t|D|})$ which is explicitly described in Theorem \ref{thm:HeatTr}. Then, for every $a\in O_A$ it holds that 
$$a |D|^{-2}\mathrm{e}^{-\t_c |D|}\in \mathcal{M}_{1,\infty}(L^2(G_A,\mu_{G_A})),$$
and 
$$\varphi_A(a)=\frac{2\t_c}{C}\Tr_{\omega}(a |D|^{-2}\mathrm{e}^{-\t_c |D|}).$$ 
\end{thm}

In particular, we note that $a |D|^{-2}\mathrm{e}^{-\t_c |D|}$ is Dixmier measurable for any $a\in O_A$.

\begin{proof}
By a density argument we can assume that $a\in C_c^{\infty}(G_A)$. We denote $T:=|D|^{-2}\mathrm{e}^{-\t_c |D|}\geq 0$. Recall that the spectrum of $|D|$ is bounded below by some $\lambda_0>0$. We aim to show that 
\begin{equation}
\label{eq:Dixmier1}
\lim_{t\to 0^+} t \Tr(aT^{1+t})= \frac{C}{2\t_c}\varphi_A(a).
\end{equation}
From \eqref{eq:Dixmier1} (with $a=1$ and \cite[Proposition 2.17]{gaysuk} we can conclude that $T\in \mathcal{M}_{1,\infty}$. From \cite[Corollary 8.6.9]{LSZ}, see also \cite[Theorem 3.4]{gaysuk}, the limit \eqref{eq:Dixmier1} equals $\Tr_{\omega}(aT)$. By continuity, this extends to all $a\in O_A$.

Following Theorems \ref{thm:HeatTr} and \ref{thm:HeatTrwitha}, there is $C_a\in \mathbb C$ such that 
\begin{equation}
\label{eq:Dixmier2}
G_a(t):=\Tr(a\mathrm{e}^{-(\t_c+t)|D|})=C_a t^{-3}+O(t^{-2}),\qquad \text{as }\,\,t\to 0^+.
\end{equation}
In particular, for $a=1$ we have that $C_1$ equals the constant $C$ in the statement. Take $t>0$ and from the Mellin transform we get 
\begin{equation}
\label{eq:Dixmier3}
|D|^{-2-2t}\mathrm{e}^{-(1+t)\t_c|D|}=\frac{1}{\Gamma(2+2t)}\int_{0}^{\infty}x^{1+2t}\mathrm{e}^{-(x+(1+t)\t_c)|D|} \d x.
\end{equation}
Multiplying \eqref{eq:Dixmier3} by $a$ from the left and then taking the trace, it yields 
\begin{equation}
\label{eq:Dixmier4}
t\Tr(aT^{1+t})= \frac{t}{\Gamma(2+2t)}\int_{0}^{\infty}x^{1+2t}G_a(x+t\t_c)\d x.
\end{equation}
Consider some $x_0>0$ and decompose $t\Tr(aT^{1+t})=I_1(t)+I_2(t)$, with 
\begin{align*}
I_1(t)&:=\frac{t}{\Gamma(2+2t)}\int_{0}^{x_0}x^{1+2t}G_a(x+t\t_c)\d x,\\
I_2(t)&:=\frac{t}{\Gamma(2+2t)}\int_{x_0}^{\infty}x^{1+2t}G_a(x+t\t_c)\d x. 
\end{align*}

First, we claim that $\lim_{t\to 0^+}I_2(t)=0$. Indeed, for $x\geq x_0$ we have 
\begin{align*}
|G_a(x+t\t_c)|&\leq \|a\| \|\mathrm{e}^{-(x+(t+1)\t_c)|D|}\|_1\\
&\leq \|a\|\|\mathrm{e}^{-(\t_c+\frac{x_0}{2})|D|}\|_1 \|\mathrm{e}^{-(x+t\t_c-\frac{x_0}{2})|D|}\| \\
&\leq \|a\|\|\mathrm{e}^{-(\t_c+\frac{x_0}{2})|D|}\|_1 \mathrm{e}^{-\frac{\lambda_0}{2} x}.
\end{align*} 
Hence, the integral term in $I_2$ is uniformly bounded for small $t>0$, and the claim follows. Now, we focus on $I_1(t)$ and assume that $x_0>0$ is small enough so that, for small $t>0$, the asymptotics \eqref{eq:Dixmier2} hold for $G_a(x+t\t_c)$ when $0\leq x\leq x_0.$ That is, $$G_a(x+t\t_c)=C_a(x+t\t_c)^{-3}+R_a(x+t\t_c),$$ with $R_a(x+t\t_c)=O((x+t\t_c)^{-2}).$ In particular, we can decompose $I_1(t)=I_{11}(t)+I_{12}(t)$, with 
\begin{align*}
I_{11}(t)&:=\frac{tC_a}{\Gamma(2+2t)}\int_{0}^{x_0}x^{1+2t}(x+t\t_c)^{-3}\d x,\\
I_{12}(t)&:=\frac{t}{\Gamma(2+2t)}\int_{0}^{x_0}x^{1+2t}R_a(x+t\t_c)\d x.
\end{align*}
It is straightforward to check that $$\lim_{t\to 0^+} I_{11}(t)= \frac{C_a}{2\t_c},\qquad \lim_{t\to 0^+} I_{12}(t)=0.$$ From Theorem \ref{thm:HeatTrwitha} we have $C_a=C\varphi_A(a)$, hence \eqref{eq:Dixmier1} follows, and the proof is complete.
\end{proof}

\subsection{Local heat trace asymptotics for the positive part of $D$}
\label{localIII}

We now study the Heisenberg flow of $O_A$ associated to $D$. Specifically, for $T\in O_A$ consider the $\mathbb R$-flow $\sigma_t(T):=\mathrm{e}^{itD}T\mathrm{e}^{-itD}$, and the saturation $O_{A,D}$ of $O_A$ under it, namely $$O_{A,D}:= C^*(\bigcup_{t\in \mathbb R} \sigma_t(O_A)).$$ Using Theorem \ref{thm:O_A_action} and the closed formula for the eigenfunctions in Proposition \ref{prop:eigenfunctions} we derive the following. Let also $P_D$ denote the non-negative spectral projection of $D$.

\begin{prop}
\label{prop:KMS}
We have that $\Tr(P_D\mathrm{e}^{-tD})<\infty$ if and only if $t>\t_c$. Specifically, 
$$\lim_{t\to \t_c^+} (t-\t_c)\Tr(P_D\mathrm{e}^{-tD})=\frac{\t_c}{\lambda_A^2 \log \lambda_A} \sum_{j=1}^{N} \mathfrak{c}(j)v_j.$$ 
Moreover, the state $\tilde{\varphi}_A$ on $O_{A,D}$ given by 
$$\tilde{\varphi}_A(T):= \lim_{t\to \t_c^+} \frac{\Tr(P_DT\mathrm{e}^{-tD})}{\Tr(P_D\mathrm{e}^{-tD})}$$ 
is a $\t_c$-KMS-state for $\sigma_t$, and its restriction on $O_A$ coincides with the KMS-state $\varphi_A$ of the gauge action.
\end{prop}

We note that results similar to Proposition \ref{prop:KMS} were studied in more detail in \cite{GRU}.

\begin{proof}
For brevity, set $\zeta:= \frac{ \log \lambda_A}{\t_c}$. Given $\alpha,\beta  \in V_A$ such that $A_{\alpha_{|\alpha|},\beta_{|\beta|}}=1,$ if $\alpha,\beta\neq {\o}$, we aim to estimate $\Tr(P_DS_{\alpha \beta_{|\beta|}} S_{\beta}^* \mathrm{e}^{-tD}).$ Notice that it suffices to work with such elements $T_{\alpha,\beta}:=S_{\alpha \beta_{|\beta|}} S_{\beta}^*$ as they span the dense $*$-subalgebra $C_c^{\infty}(G_A)$ of $O_A$. 

To this end, for any $\alpha'.\beta' \in \tilde{V}_A$, using the Cuntz--Krieger relations and Remark \ref{rem: char_functions}, we obtain that 

$$T_{\alpha,\beta}\mathrm{e}_{\alpha'.\beta'}=
\begin{cases}
\mathrm{e}_{\alpha \beta_{|\beta|}\delta.\beta'}, & \text{if  } \alpha'=\beta \delta\,\, \text{for a unique } \delta\in V_A\\
\mathrm{e}_{\alpha. \delta}, & \text{if  } \beta=\alpha'\delta\,\, \text{for a unique } \delta\in V_A\setminus \{{\o}\}\,\, \text{with } \delta_1=\beta'. 
\end{cases}$$ 
Moreover, for $t>\t_c$ we have
\begin{align*}
\Tr(P_DT_{\alpha,\beta}\mathrm{e}^{-tD})&= \sum_{\gamma\in I_A} \left(\langle P_D T_{\alpha,\beta} \mathrm{e}^{-tD} \mathrm{e}_{\gamma},\mathrm{e}_{\gamma}\rangle +\sum_{\nu,j} \langle P_D T_{\alpha,\beta} \mathrm{e}^{-tD} \mathrm{e}_{\gamma,\nu,j},\mathrm{e}_{\gamma,\nu,j}\rangle \right)\\
&= \sum_{\alpha'.\beta'\in \tilde{V}_A} \langle P_D T_{\alpha,\beta} \mathrm{e}^{-tD} \mathrm{e}_{\alpha'.\beta'},\mathrm{e}_{\alpha'.\beta'}\rangle\\
&=\langle T_{\alpha,\beta}\mathrm{e}^{-tD}e_{\hat{\beta}.\beta_{|\beta|}} ,e_{\hat{\beta}.\beta_{|\beta|}}\rangle + \sum_{\beta'=1}^{N} A_{\beta_{|\beta|},\beta'} \langle T_{\alpha,\beta}\mathrm{e}^{-tD}\mathrm{e}_{\beta.\beta'}, \mathrm{e}_{\beta.\beta'} \rangle\\
&+ \sum_{\beta'=1}^{N} \sum_{\delta \in V_A\setminus \{{\o}\}} B_{\beta_{|\beta|},\delta_1} A_{\delta_{|\delta|},\beta'} \langle T_{\alpha,\beta}\mathrm{e}^{-tD}\mathrm{e}_{\beta \delta.\beta'}, \mathrm{e}_{\beta \delta.\beta'} \rangle,
\end{align*}
where $B_{\beta_{|\beta|},\delta_1}=A_{\beta_{|\beta|},\delta_1}$ if $\beta\neq {\o}$, and $B_{\beta_{|\beta|},\delta_1}=1$ if $\beta={\o}$. As a result, if $\alpha\neq \hat{\beta}$ then 
\begin{equation}\label{eq:KMS0}
\Tr(P_DT_{\alpha,\beta}\mathrm{e}^{-tD})=0,
\end{equation}
and if $\beta\neq {\o}$ then
\begin{align}\label{eq:KMS1}
\Tr(P_DT_{\hat{\beta},\beta}\mathrm{e}^{-tD})&=\mathrm{e}^{-t|\beta|_c}+\sum_{\beta'=1}^{N} A_{\beta_{|\beta|},\beta'}\mathrm{e}^{-t(|\beta|_c+\zeta)}+\sum_{\beta'=1}^{N} \sum_{\delta \in V_A\setminus \{{\o}\}} A_{\beta_{|\beta|},\delta_1} A_{\delta_{|\delta|},\beta'} \mathrm{e}^{-t(|\beta|_c+|\delta|_c+\zeta)}\\ \nonumber
&=\mathrm{e}^{-t|\beta|_c}+\sum_{\beta'=1}^{N} A_{\beta_{|\beta|},\beta'}\mathrm{e}^{-t(|\beta|_c+\zeta)}+\mathrm{e}^{-t(|\beta|_c+\zeta)}\sum_{\beta'=1}^{N}\sum_{i,j=1}^{N}A_{\beta_{|\beta|},i}A_{j,\beta'}F_c(t,i,j)\\ \nonumber
&=\mathrm{e}^{-t|\beta|_c}+\sum_{\beta'=1}^{N} A_{\beta_{|\beta|},\beta'}\mathrm{e}^{-t(|\beta|_c+\zeta)}+\mathrm{e}^{-t(|\beta|_c+\zeta)}\sum_{i,j=1}^{N}A_{\beta_{|\beta|},i}\mathfrak{c}(j)F_c(t,i,j).
\end{align}
From Lemma \ref{lem:Potential_est} we then obtain that
\begin{align}\label{eq:KMS2}
\lim_{t\to \t_c^+} (t-\t_c)\Tr(P_DT_{\hat{\beta},\beta}\mathrm{e}^{-tD})&=\zeta^{-1}\lambda_A^{-|\beta|-2}\left(\sum_{j=1}^{N}\mathfrak{c}(j)v_j\right) \left(\sum_{i=1}^{N} A_{\beta_{|\beta|},i}u_i\right)\\ \nonumber
&= \zeta^{-1}\lambda_A^{-|\beta|-2}\left(\sum_{j=1}^{N}\mathfrak{c}(j)v_j\right) \lambda_Au_{\beta_{|\beta|}}\\ \nonumber
&= \zeta^{-1}\lambda_A^{-|\beta|-1}u_{\beta_{|\beta|}}\left(\sum_{j=1}^{N}\mathfrak{c}(j)v_j\right).
\end{align}

If $\beta={\o}$, from \eqref{eq:KMS1} we get 
\begin{equation*}
\Tr(P_D\mathrm{e}^{-tD})=N\mathrm{e}^{-t\zeta}+\sum_{\beta'=1}^{N} \sum_{\delta \in V_A\setminus \{{\o}\}} A_{\delta_{|\delta|},\beta'} \mathrm{e}^{-t(|\delta|_c+\zeta)}.
\end{equation*}
It is also clear that $\Tr(P_D\mathrm{e}^{-tD})<\infty$ if and only if $t>\t_c$ and 
\begin{equation}\label{eq:KMS3}
\lim_{t\to \t_c^+} (t-\t_c)\Tr(P_D\mathrm{e}^{-tD})=\zeta^{-1} \lambda_A^{-2}\left(\sum_{j=1}^{N}\mathfrak{c}(j)v_j\right).
\end{equation}
Then, from \eqref{eq:KMS0}, \eqref{eq:KMS2} and \eqref{eq:KMS3} we get that $\tilde{\varphi}_A$ on $O_A$ is the KMS-state $\varphi_A$. 

Finally, for every $t\in \mathbb R$ and $T\in O_A$ it holds 
$\varphi_A(\sigma_t(T))=\varphi_A(T),$ and from \cite[Theorem 1]{GRU} we get that $\varphi_A$ is a $\t_c$-KMS-state for $\sigma_t$.
\end{proof}

\section{The local Weyl law and relations to quantum ergodicity}

We can compute the Weyl law in the case of regular graphs, and in fact even a local Weyl law. This will lead us to an averaged version of quantum ergodicity, as discussed in Remark \ref{qedisc} above. Here, we follow the convention of, for instance, \cite[Section 2]{zelditchqe} to call the asymptotic behaviour of $\Tr(a\chi_{[0,\theta]}(|D|))$ a local Weyl law. 

\subsection{Weyl Law for regular graphs}
We use $\Tr(\mathrm{e}^{-t|D|})$ to obtain a Weyl law for $|D|$. Our methods are restricted to regular strongly connected graphs since we need an explicit heat trace. The key tool is Perron's inversion formula as the Tauberian theorem does not apply, since there are infinitely many poles along imaginary lines in the heat trace.

\begin{thm}
\label{thm:Weyl_law}
Assume that the row-sums and column-sums of $A$ are equal to $d\geq 2$. Then, $\t_c= \log(d)c^{-1}$ with $c:=N^{-1}(d-1)$, and for $t>\t_c$ we have $$\Tr(\mathrm{e}^{-t|D|})= N\frac{(\mathrm{e}^{ct}+(d-1)\mathrm{e}^{-tN^{-1}}-d)(\mathrm{e}^{ct}-d\mathrm{e}^{-ct})}{(\mathrm{e}^{ct}-d)^3}.$$ The function $t\mapsto \Tr(\mathrm{e}^{-t|D|})$ extends to a meromorphic function on $\mathbb C$ with poles forming $$\mathfrak{T}_A=\{\t_c(k):=\t_c+2\pi c^{-1} k i: k\in \mathbb Z\}.$$ 
Also, let $(\theta_j)_{j\geq 0}$ denote the eigenvalues of $|D|$ in non-decreasing order, counting multiplicity. Then, the counting function $\mathcal{N}(\theta):=\#\{\theta_j\leq \theta:j\geq 0\}$ satisfies 
$$\mathcal{N}(\theta)=c_{d,N}\theta^2 d^{\lfloor \frac{\theta N-1}{d-1}\rfloor}+O(\theta \mathrm{e}^{\theta\t_c}),$$ 
where $\lfloor x\rfloor\in \Z$ denotes the floor function of $x\in \R$, and 
$$c_{d,N}=\frac{N^3}{2(d-1)d^{2}}.$$
\end{thm}

\begin{proof}
Let $t>\t_c$ and $j\in \{1,\ldots, N\}$, and define $S_j(t):=\sum_{\alpha.\beta\in I_A, \beta_{|\beta|}=j}\mathrm{e}^{-t(|\alpha|_c+|\beta|_c)}.$ From the proof of Theorem \ref{thm:HeatTr} we have that 
\begin{equation}\label{eq:Weyl_0}
\Tr(\mathrm{e}^{-t|D|})=\sum_{j=1}^{N} \Tr(\mathrm{e}^{-t\Delta_{C(j)}}) S_j(t).
\end{equation}

The out-regularity of the graph gives $\t_c= \log(d)c^{-1}$,  $c:=N^{-1}(d-1)$ (Example \ref{ex:out-reg}), and allows to compute $\Tr(\mathrm{e}^{-t\Delta_{C(j)}})$ explicitly. Specifically, for $j\in \{1,\ldots, N\}$ and from \eqref{eq:eigenvalues} we see that for every $\nu \in jV_{A}$, the eigenvalue $\lambda_j^A(\nu)$ of $\Delta_{C(j)}$ is given by $$\lambda_j^A(\nu)=dN^{-1}+c(|\nu|-1).$$ Also, it has multiplicity $(d-1)d^{n-1}$. As a result, we obtain that 
\begin{equation*}
\Tr(\mathrm{e}^{-t\Delta_{C(j)}})= \frac{(\mathrm{e}^{ct}+(d-1)\mathrm{e}^{-tN^{-1}}-d)}{\mathrm{e}^{ct}-d}.
\end{equation*}
The in-regularity of the graph allows to compute $S_j(t)$ explicitly. First, one calculates that for each $k\in \{1,\ldots,N\}$ the function $F_c(t,k)$ from \eqref{eq:HeatTr0} is given by $$F_c(t,k)= \frac{1}{\mathrm{e}^{ct}-d}.$$ Then, following the proof of Theorem \ref{thm:HeatTr} we get 
\begin{equation*}
S_j(t)= \frac{\mathrm{e}^{ct}-d\mathrm{e}^{-ct}}{(\mathrm{e}^{ct}-d)^2}.
\end{equation*}
From \eqref{eq:Weyl_0} it follows that $$\Tr(\mathrm{e}^{-t|D|})= N\frac{(\mathrm{e}^{ct}+(d-1)\mathrm{e}^{-tN^{-1}}-d)(\mathrm{e}^{ct}-d\mathrm{e}^{-ct})}{(\mathrm{e}^{ct}-d)^3}.$$

It is immediate to see that the poles are $\mathfrak{T}_A=\{\t_c(k):=\t_c+2\pi c^{-1} k i: k\in \mathbb Z\}.$ Also, denote by $c_{-3}(k),c_{-2}(k)$ and $c_{-1}(k)$ the Laurent coefficients of $\Tr(\mathrm{e}^{-t|D|})$ at $t=\t_c(k)$. Setting $s_k:=d^{-(d-1)^{-1}}\mathrm{e}^{-2\pi (d-1)^{-1}ki}$, elementary calculations yield that 
\begin{align}\label{eq:Weyl_1}
c_{-3}(k)&= \frac{N^4}{d^3(d-1)}s_k,\\
\nonumber c_{-2}(k)&= \frac{N^3}{2d^3(d-1)}(2d+(3-d)s_k),\\
\nonumber c_{-1}(k)&= \frac{N^2}{2d^3(d-1)}(4d+(2-5d)s_k).
\end{align}
Now from Perron's theorem \cite[Theorem 13]{HR} we get that the counting function $\mathcal{N}(\theta)$ for $\theta >0$ is given by a principal value integral formula. Namely, for an arbitrary $\eta_+>\t_c$,
\begin{equation}
\label{eq:Weyl_2}
\mathcal{N}(\theta)=\lim_{T\to \infty}\frac{1}{2\pi i}\int_{\eta_+-Ti}^{\eta_+ + Ti} \Tr(\mathrm{e}^{-t|D|}) \mathrm{e}^{\theta t} t^{-1} \d t <\infty.
\end{equation}
Set $Z_{\theta}(t):= \Tr(\mathrm{e}^{-t|D|}) \mathrm{e}^{\theta t} t^{-1}$ for brevity, and consider the sequence $(T_n)_{n\geq 0}$ given by $T_n:=(n+2^{-1})2\pi c^{-1}$. Also, consider some $\eta_-<\t_c$ and the rectangle $R_n$ centred around $\text{Re}(t)=\t_c$ that is oriented counter-clockwise, with sides 
\begin{align*}
R_{n,1}&=\{\eta_+ + y i : y\in \mathbb R, |y|\leq T_n\},\\
R_{n,2}&=\{\eta + T_n i : \eta_- \leq \eta \leq \eta_+\},\\
R_{n,3}&=\{\eta_- + y i : y\in \mathbb R, |y|\leq T_n\},\\
R_{n,4}&=\{\eta - T_n i : \eta_- \leq \eta \leq \eta_+\}.
\end{align*}
Let $I_{n,j}(\theta):=\frac{1}{2\pi i}\int_{R_{n,j}} Z_{\theta}(t)\d t$ for $j=1,2,3,4$, and then the Residue theorem asserts that 
\begin{equation}\label{eq:Weyl_3}
I_{n,1}(\theta) =\sum_{|\Im(\t_c(k))|<T_n} \text{Res}_{t=\t_c(k)}Z_{\theta}(t) - (I_{n,2}(\theta)+I_{n,4}(\theta)) - I_{n,3}(\theta).
\end{equation}

The first observation is that 
\begin{equation}\label{eq:Weyl_4}
\lim_{n\to \infty} I_{n,1}(\theta)=\mathcal{N}(\theta).
\end{equation}
Moreover,
\begin{equation}\label{eq:Weyl_5}
\lim_{n\to \infty} (I_{n,2}(\theta)+I_{n,4}(\theta))=0.
\end{equation}
Indeed, for $t\in R_{n,2}\cup R_{n,4}$ we have $|\Tr(\mathrm{e}^{-t|D|})|\leq M_{\eta_-,\eta_+}$, where the constant $M_{\eta_-,\eta_+}>0$ is independent of $n\geq 0$. This is because $\mathrm{e}^{cT_ni}=-1$, which gives $|\mathrm{e}^{ct}-d|\geq d$. Therefore, there is an $M_{\eta_-,\eta_+,\theta}>0$ such that $$|I_{n,2}(\theta)+I_{n,4}(\theta)|\leq \frac{M_{\eta_-,\eta_+,\theta}}{T_n}.$$ Similarly, for $t\in R_{n,3}$ it holds that $|\Tr(\mathrm{e}^{-t|D|})|\leq M_{\eta_-}$, where the constant $M_{\eta_-}>0$ is independent of $n\geq 0$. The reason is that $|\mathrm{e}^{ct}-d|\geq d(1-\mathrm{e}^{c(\t_c-\eta_-)})>0$. However, here we have 
\begin{equation}\label{eq:Weyl_6}
|I_{n,3}(\theta)|\leq M_{\eta_-}\mathrm{e}^{\theta \eta_-} \int_{-T_n}^{T_n}\frac{1}{\sqrt{\eta_-+y^2}}\d y = O(\mathrm{e}^{\theta \t_c} \log (T_n)).
\end{equation} 
Moving now to the residues, through elementary computations we derive that for $k\in \mathbb Z$, 

\small
\begin{equation}\label{eq:Weyl_7}
\mathrm{Res}_{t=\t_c(k)} Z_{\theta}(t)
= \mathrm{e}^{\theta \t_c(k)} \Biggl(
\frac{c_{-1}(k)}{\t_c(k)}
+ c_{-2}(k)\Bigl(\frac{\theta}{\t_c(k)}-\frac{1}{\t_c(k)^2}\Bigr)
+ \frac{c_{-3}(k)}{2}\Bigl(\frac{\theta^2}{\t_c(k)}-\frac{2\theta}{\t_c(k)^2}+\frac{2}{\t_c(k)^3}\Bigr)
\Biggr).
\end{equation}
\normalsize
Consequently, since $|\t_c(k)|\geq 2\pi c^{-1}k$ we get
\begin{equation}\label{eq:Weyl_8}
\mathrm{Res}_{t=\t_c(k)} Z_{\theta}(t)
= \frac{c_{-3}(k)\theta^2+2c_{-2}(k)\theta+2c_{-1}(k)}{2\,\t_c(k)}\mathrm{e}^{2\pi \theta c^{-1}ki}\,\mathrm{e}^{\theta \t_c}
+ E(k,\theta),
\end{equation}
where the rest terms 
\begin{equation}\label{eq:Weyl_9}
E(k,\theta)=O(\theta \mathrm{e}^{\theta \t_c}k^{-2})
\end{equation}
Now consider the function 
$$F(x)=\lim_{n\to \infty}\sum_{k=-n}^n \frac{\mathrm{e}^{2\pi i k x}}{\log(d)+2\pi i k}.$$
By computations found in standard text books on Fourier series, see for instance \cite[Page 28, Table 1, Item 19]{follandfourier}, $F$ is a pointwise convergent Fourier series on $x\in \R\setminus \Z$ converging to the expression 
$$F(x)=\frac{d^{-\{x\}+1}}{d-1},$$ 
for $\{x\}\in (0,1)$ denoting the fractional part of $x\in \R\setminus \Z$, extended to a periodic function on $\R$ by $F(0)=\frac{1}{2}(F(0+)+F(0-))$.

It follows from the fact that $F$ is a pointwise convergent Fourier series that 
$$\tilde{G}(\theta):=\lim_{n\to \infty}\sum_{k=-n}^n  \frac{c_{-3}(k)\theta^2+2c_{-2}(k)\theta+2c_{-1}(k)}{2\,\t_c(k)}\mathrm{e}^{2\pi \theta c^{-1}k i},$$
is a pointwise convergent Fourier series converging to an expression of the form 
\begin{equation}
\label{tildeg}
\tilde{G}(\theta)=\frac{N^3}{2 d^{\frac{3d-2}{d-1}}}\theta^2 F\left(\frac{\theta N-1}{d-1}\right)+O(\theta).
\end{equation}
Therefore, substituting \eqref{eq:Weyl_8} to \eqref{eq:Weyl_3} and taking $n\to \infty$, we obtain 
$$\mathcal{N}(\theta)=\tilde{G}(\theta)\mathrm{e}^{\theta\t_c} + \lim_{n\to \infty} \left( \sum_{|\Im(\t_c(k))|<T_n} E(k,\theta) - I_{n,3}(\theta)\right),$$ 
meaning that the limit exists as both $\mathcal{N}(\theta)$ and $\tilde{G}(\theta)\theta^2 \mathrm{e}^{\theta\t_c}$ are finite. Finally, using the asymptotics in \eqref{eq:Weyl_6}, \eqref{eq:Weyl_9} and \eqref{tildeg} we complete  the proof by concluding 
\begin{align*}
\mathcal{N}(\theta)=&\tilde{G}(\theta)\mathrm{e}^{\theta\t_c} +O(\theta\mathrm{e}^{\theta \t_c})=\\
=&\frac{N^3}{2 d^{\frac{3d-2}{d-1}}}\theta^2 F\left(\frac{\theta N-1}{d-1}\right)\mathrm{e}^{\theta \t_c}+O(\theta\mathrm{e}^{\theta \t_c})=\\
=&\frac{N^3}{2 d^{\frac{3d-2}{d-1}}}\frac{d^{-\left\{\frac{\theta N-1}{d-1}\right\}+1}}{d-1}d^{\frac{\theta N-1}{d-1}+\frac{1}{d-1}}+O(\theta\mathrm{e}^{\theta \t_c})\\
=&\frac{N^3}{2(d-1)d^{\frac{3d-2}{d-1}-\frac{d}{d-1}}}\theta^2 d^{\lfloor \frac{\theta N-1}{d-1}\rfloor}+O(\theta \mathrm{e}^{\theta\t_c}).
\end{align*}
\end{proof}

\begin{remark}
The roots of $\Tr(\mathrm{e}^{-t|D|})$ are $$R=\{2^{-1}\t_c+\pi c^{-1}ki:k\in \mathbb Z\}\bigcup \{t\in \mathbb C: y=\mathrm{e}^{tN^{-1}}, y^d-dy+(d-1)=0\}.$$
\end{remark}

In order to understand the asymptotics of Theorem \ref{thm:Weyl_law}, we view the space of admissible words $V_A$ as a rooted tree with root ${\o}$, and equip it with the metric $\rho:V_A\times V_A\to \mathbb N\cup \{0\}$ defined as $\rho(v,w):=|v|+|w|-2|v\wedge w|,$ where $v\wedge w$ is the first common ancestor of $v,w$ in $V_A$. This turns $(V_A,\rho)$ into a Gromov hyperbolic space with product $$(v,w)_z:=\frac{1}{2}(\rho(v,z)+\rho(w,z)-\rho(v,w)),$$ where the geodesic triangles are $0$-thin and hence are tripods. Here we focus on the ones starting from $\{{\o}\}$, forming the set $$\mathrm{Trip}_{A,{\o}}=\{(v,w,z)\in (V_A\setminus \{{\o}\})^{\times 3}: zv,zw\in V_A,\,\, v_1\neq w_1\}.$$ The perimeter of such $(v,w,z)$ is then $\mathrm{Per}(v,w,z)=2(|v|+|w|+|z|).$

\begin{cor}
Assume that the row-sums and column-sums of $A$ are equal to $d\geq 2$. The counting function $\Pi(\theta, A^t):=\# \{(v,w,z)\in \mathrm{Trip}_{A^t,{\o}}: \mathrm{Per}(v,w,z)<\theta\}$ associated with $A^t$ satisfies $$\Pi(\theta, A^t)= C_{d,N}\theta^2d^{\lfloor \frac{\theta}{2} \rfloor}+ O(\theta d^{\frac{\theta}{2}}),\qquad C_{d,N}=\frac{N}{8d}.$$ Consequently, the same asymptotics hold also for $\Pi(\theta, A)$.
\end{cor}

\begin{proof}
The eigenvalues of $D$ are parametrised (without considering their multiplicity) by the set $$Q_{A}:=\{(\alpha,\hat{\beta},\tilde{\nu})\in V_A\times V_A\times (V_A\setminus \{{\o}\}): \alpha.\beta \in I_A, \,\,\tilde{\nu}\in \beta_{|\beta|}V_A\},$$ and they are of the form $dN^{-1}+c(|\alpha|+|\hat{\beta}|+|\tilde{\nu}|)$ with $c=N^{-1}(d-1)$. In particular, the counting function $\mathcal{N}$ from Theorem \ref{thm:Weyl_law} is $$\mathcal{N}(\theta)=(d-1)\cdot  \# \{(\alpha,\hat{\beta},\tilde{\nu})\in Q_A: dN^{-1}+c(|\alpha|+|\hat{\beta}|+|\tilde{\nu}|)<\theta\}.$$ Further, it is straightforward to see that $Q_A$ is in bijection with $$Q'_{A^t}:=\mathrm{Trip}_{A^t,{\o}}\bigsqcup \{(v,w,z)\in V_{A^t}^{\times 3}:z\neq {\o},\,\, zv,zw\in V_{A^t},\,\, (v={\o}\,\, \text{or } w={\o})\}.$$ As a result, we have 
$$\# \{(v,w,z)\in Q'_{A^t}:2(|v|+|w|+|z|)<\theta\}=\frac{\mathcal{N}(dN^{-1}+2^{-1}c\theta)}{d-1},$$ which implies 
\begin{align*}
\Pi(\theta,A^t)&=\frac{\mathcal{N}(dN^{-1}+2^{-1}c\theta)}{d-1}+O(d^{2^{-1} \theta })\\
&= \frac{c_{d,N}}{d-1}(dN^{-1}+2^{-1}c\theta)^2d^{\lfloor 2^{-1} \theta +1  \rfloor}+O(\theta d^{2^{-1} \theta })+O(d^{2^{-1} \theta })\\
&=\frac{c_{d,N}}{d-1}d4^{-1}c^2 \theta^2 d^{\lfloor 2^{-1} \theta \rfloor}+O(\theta d^{2^{-1} \theta })\\
&=C_{d,N}\theta^2d^{\lfloor \frac{\theta}{2} \rfloor}+ O(\theta d^{\frac{\theta}{2}}).
\end{align*}
This completes the proof.
\end{proof}

\subsection{The local Weyl law for regular graphs and quantum ergodicity}

We now turn our attention to the local Weyl law and its applications in quantum ergodicity.

\begin{lemma}
\label{thm:localWeyl_law}
Assume that the row-sums and column-sums of $A$ are equal to $d\geq 2$. Then, $\t_c= \log(d)c^{-1}$ with $c:=N^{-1}(d-1)$, and for $t>\t_c$ and $w,\nu \in V_A$ we have 
\begin{equation}
\label{loalaadwe}
\Tr(S_w S_\nu^* \mathrm{e}^{-t|D|})= \delta_{w,\nu}\mathrm{e}^{-t|w|_c}d\frac{(\mathrm{e}^{ct}+(d-1)\mathrm{e}^{-tN^{-1}}-d)(\mathrm{e}^{ct}-d\mathrm{e}^{-ct})}{(\mathrm{e}^{ct}-d)^3}+\delta_{w,\nu}f_{w}(t),
\end{equation}
for a function $f_{w}$ meromorphic in a neighborhood of the strip $\mathrm{Re}(t)\geq t_c$ with at most order two poles along the set $\mathfrak{T}_A$ (as defined in Theorem \ref{thm:Weyl_law}). The function $t\mapsto \Tr(S_w S_\nu^*\mathrm{e}^{-t|D|})$ extends to a meromorphic function on a neighborhood of the strip $\mathrm{Re}(t)\geq \t_c$ with poles forming $\mathfrak{T}_A$. Also, for the counting function $\mathcal{N}(\theta):=\#\{\theta_j\leq \theta:j\geq 0\}$ defined as in Theorem \ref{thm:Weyl_law} and $P_\theta:=\chi_{[0,\theta]}(|D|)$ the spectral projection of $|D|$, we have that  
\begin{equation}
\label{specprojas}
\Tr(P_\theta S_w S_\nu^*)=\delta_{w,\nu}dN^{-1}\mathcal{N}\left(\theta-|w|_c \right)+O(\theta \mathrm{e}^{\theta\t_c}).
\end{equation}
\end{lemma}

We will refer to the asymptotics \eqref{loalaadwe} as the local heat trace asymptotics and to the asymptotics \eqref{specprojas} as the local Weyl law.

\begin{proof}
The same proof as that of Theorem \ref{thm:Weyl_law}, using Perron's theorem summarized in the formula \eqref{eq:Weyl_2}, shows that the local heat trace asymptotics \eqref{loalaadwe} implies the local Weyl law \eqref{specprojas}. The proof of the local heat trace asymptotics \eqref{loalaadwe} follows the same lines as the proof of Theorem \ref{thm:HeatTrwitha}, and we omit this long computation.
\end{proof}

\begin{thm}
\label{averageqe}
Assume that the row-sums and column-sums of $A$ are equal to $d\geq 2$. Then for any $a\in O_A$, with $P_\theta:=\chi_{[0,\theta]}(|D|)$ denoting the spectral projection of $|D|$ and $\varphi_A$ the KMS-state on $O_A$, we have that 
$$\lim_{\theta\to \infty} \frac{\Tr(P_\theta a)}{\Tr(P_\theta)}=\varphi_A(a).$$
\end{thm}

\begin{proof}
It follows from Lemma \ref{thm:localWeyl_law} that for $a=S_w S_\nu^*$ it holds that
$$\lim_{\theta\to \infty} \frac{\Tr(P_\theta S_w S_\nu^*)}{\Tr(P_\theta)}=\lim_{\theta\to \infty}\delta_{w,\nu}\frac{d}{N}\frac{\mathcal{N}\left(\theta-|w|_c\right)}{\mathcal{N}(\theta)}=\delta_{w,\nu}d^{-|w|+1}N^{-1}=\varphi_A(S_w S_\nu^*),$$
where the second to last identity follows from Theorem \ref{thm:Weyl_law}. By linearity and density, we conclude that $\lim_{\theta\to \infty} \frac{\Tr(P_\theta a)}{\Tr(P_\theta)}=\varphi_A(a)$ for any $a\in O_A$ from an $\varepsilon/3$-argument. The last statement of the theorem now follows from Theorem \ref{thm:HeatTrwitha} and \cite[Proposition 4.1]{hekkelmc}.
\end{proof}

As discussed in Remark \ref{qedisc} above, Theorem \ref{averageqe} can be interpreted as a statement about quantum ergodicity. We refer the reader also to \cite{hekkelmc} for further context of quantum ergodicity in noncommutative geometry.

\begin{prop}
\label{aldjnadljn}
Under the assumptions and notations of Theorem \ref{averageqe}, consider the sequence of states 
$$\rho_n(a)=\langle e_n, ae_n\rangle_{L^2(G_A,\mu_{G_A})}, \; n=1,2,\ldots.$$
There is no density 1 subsequence of the integers along which $(\rho_n)_n$ converges in the weak$^*$-sense.
\end{prop}

\begin{proof}
Let $\mu_n$ be the measure on $\Omega_A$ such that for $a\in C(\Omega_A)$,
$$\rho_{n}(a)= \int_{\Omega_A} a\mathrm{d}\mu_n.$$
We will show that $(\mu_n)_n$ cannot converge in the weak$^*$-sense along any density one subsequence $(n_k)_k$. 

Consider a density one subsequence $(n_k)_k$. Since the subsequence has density one, the construction of the ON-basis $(e_n)_n$ in Subsection \eqref{subsec:ham} ensures that for any $x=x_1x_2\cdots $ from a set of full measure in $\Omega_A$, we can find a subsequence $(n_{k_l})_l$ such that $\mu_{n_{k_l}}$ is supported in the cylinder set $C(x_1x_2\cdots x_l)\subset \Omega_A$. In particular, $\lim_{l\to \infty}\mu_{n_{k_l}}=\delta_x$. Therefore, $(\mu_n)_n$ cannot converge in the weak$^*$-sense along $(n_k)_k$.
\end{proof}

\end{document}